\documentclass[a4paper, 11pt]{amsart}
\usepackage[utf8]{inputenc}
\usepackage[T1]{fontenc}
\usepackage[english]{babel}
\usepackage{graphicx}
\usepackage{stmaryrd}
\usepackage{tikz}
\usepackage{tikz-cd} 
\usepackage[all]{xy}
\usepackage{comment}
\usepackage{bm}
\usepackage{comment}
\usepackage{amsmath,amsfonts,amssymb}
\usepackage[a4paper]{geometry}
\usepackage{amsthm}
\usepackage{float}
\usepackage{enumitem}
\usepackage{hyperref}
\usepackage{xcolor}
\hypersetup{
    colorlinks=true,
    linkcolor={blue},
    citecolor={blue},
    urlcolor={blue}}
\newtheorem{te}{Theorem}[section]

\newtheorem{prop}[te]{Proposition}

\newtheorem{co}[te]{Corollary}
\newtheorem{conj}[te]{Conjecture}
\newtheorem{qu}[te]{Question}
\newtheorem{lemme}[te]{Lemma}
\newtheorem{claim}[te]{Claim}

\theoremstyle{definition}

\newtheorem{de}[te]{Definition}
\newtheorem{ex}[te]{Example}

\theoremstyle{remark}
\newtheorem{rque}[te]{Remark}

\newlength{\plarg}
\usetikzlibrary{cd}
\calclayout

\title{Simple sharply-2 transitive groups}
\author{Simon Andr\'e and Katrin Tent}
\date{\today}
\begin{document}

\begin{minipage}{\linewidth}

\maketitle

\vspace{0mm}

\begin{abstract}
We construct simple sharply 2-transitive groups. Our result answers an open question of Peter Neumann. In fact, we prove that every sharply 2-transitive group $G$ of characteristic 0 embeds into a simple sharply 2-transitive group.
\end{abstract}

	
\end{minipage}

\vspace{3mm}

\section{Introduction}


\thispagestyle{empty}

Let $G$ be a group acting on a set $X$. The action is said to be \emph{sharply $n$-transitive} if for any two $n$-tuples $(x_1,\ldots,x_n)$ and $(y_1,\ldots,y_n)$ of distinct elements of $X$, there exists a unique element $g$ in $G$ such that $g\cdot x_i=y_i$ for every $1\leq i\leq n$. The group $G$ is called \emph{sharply $n$-transitive} if there exists a set $X$ on which $G$ acts sharply $n$-transitively. Every group is sharply 1-transitive since it acts sharply 1-transitively on itself by left multiplication, but the situation changes drastically when one considers sharply multiply transitive groups, i.e.\ when $n\geq 2$. 


\smallskip

For $n\geq 4$, the restrictions on the groups that can act sharply $n$-transitively on some set are extremely severe: in 1872, Jordan proved in \cite{Jor} that the only finite sharply $n$-transitive groups for $n\geq 4$ are the symmetric groups $S_n$ and $S_{n+1}$, the alternating group $A_{n+2}$, and the Mathieu groups $M_{11}$ and $M_{12}$, respectively for $n=4$ and $n=5$. The classification was completed in \cite[Chapitre IV, Th\'eor\`eme I]{Tits1} by Tits, who proved that there is no infinite sharply $n$-transitive group for $n\geq 4$ (see also \cite[Theorem 1.11]{Cam}). Stronger results were proved by Hall in \cite{Hall}, then by Yoshizawa in \cite{Yos}: if a group acts 4-transitively on an infinite set, then the stabilizer of any 4-tuple must be infinite.

\smallskip

For $n=2$ and $n=3$, the finite sharply $n$-transitive groups were completely classified by Zassenhaus in \cite{Z1} and \cite{Z2}. However, in contrast with the case where $n\geq 4$, there are infinite examples: given a field $K$, the affine group $\mathrm{AGL}_1(K)\simeq K\rtimes K^{\ast}$ acts sharply 2-transitively on $K$, and the projective linear group $\mathrm{PGL}_2(K)$ acts sharply 3-transitively on the projective line. 

\smallskip






All the finite sharply $2$-transitive groups classified by Zassenhaus, as well as $\mathrm{AGL}_1(K)$, are of the form $G=A\rtimes H$ where $A$ denotes a non-trivial proper abelian subgroup of $G$ (equivalently, $G\simeq K\rtimes K^{\ast}$ for some near-field $K$). Until recently, it was a famous open question whether all sharply $2$-transitive groups must be of this form. In \cite[Section 1.13]{Cam}, Cameron wrote: 

\smallskip

\begin{quote}
``The depth of our ignorance about infinite sharply 2-transitive groups is shown by the fact that it is not even known whether or not every such group has a regular normal subgroup.''
\end{quote}

\let\thefootnote\relax\footnote{Funded by the Deutsche Forschungsgemeinschaft (DFG, German Research Foundation) under Germany’s Excellence Strategy EXC 2044–390685587, Mathematics Münster: Dynamics–Geometry–Structure and by CRC 1442 Geometry: Deformations and Rigidity.}


The question also appears in the introduction of Kerby's book \cite{Kerby} in an equivalent form (``Is every near-domain a near-field?''), or in \cite[Problem 11.52]{unsolved}. Many results for specific classes of groups were proved in the direction of a positive answer (see for instance \cite{Tits2}, \cite{Kerby}, \cite{BN}, \cite{Tur}, \cite{GG2}, \cite{GMS}, and more recently \cite{ABW}, \cite{CT1}, \cite{CT2} in the context of groups of finite Morley rank), but the question was finally answered in the negative by Rips, Segev and Tent in \cite{RST} and by Rips and Tent in \cite{RT} (see also \cite{TZ} and \cite{GG}). In these papers, the authors constructed the first sharply $2$-transitive groups without a non-trivial proper abelian subgroup. Then, Tent exhibited in \cite{Tent1} the first infinite sharply $3$-transitive groups that do not look like $\mathrm{PGL}_2(K)$ (in the sense that the stabilizer of a point does not have a non-trivial proper abelian subgroup).

\smallskip

However, none of the sharply $2$-transitive groups constructed in the papers mentioned above are simple (for instance, the groups that appear in \cite{TZ} map onto the free group of rank 2, and hence they have plenty of normal subgroups), and the following question, posed by Peter Neumann, remained open.

\begin{qu}[P. Neumann]
Are there simple sharply $2$-transitive groups? 
\end{qu}

This question is natural since simple groups are furthest from having a non-trivial proper normal abelian subgroup. We answer this question positively. 

\begin{te}\label{existence}
There exist simple sharply 2-transitive groups.
\end{te}

\begin{rque}These groups are necessarily infinite in view of the classification of Zassenhaus mentioned above, but it is worth pointing out that the non-existence of finite simple sharply $2$-transitive groups follows more directly from a famous theorem of Frobenius (which has an easy proof in the case of finite sharply 2-transitive groups), much older than the result of Zassenhaus: if a finite group $G$ contains a malnormal subgroup $H$ (see Definition \ref{malnormal}), then there exists a normal subgroup $N$ of $G$ such that $G=N\rtimes H$. This result, combined with the fact that the centralizer of an element of order 2 in a sharply $2$-transitive group is a proper malnormal subgroup, shows that there is no finite simple sharply $2$-transitive group. The existence of an element of order 2 is an easy fact, which is explained below. Last, it is worth noting that Frobenius' theorem fails in the infinite case.\end{rque}

\begin{rque}Note that there are well-known examples of infinite simple sharply 3-transitive groups, namely $\mathrm{PSL}_2(K)$ when $K$ is a field in which every element is a square (as in this case $\mathrm{PSL}_2(K)$ and $\mathrm{PGL}_2(K)$ coincide). To our knowledge, these are essentially the only examples, and thus the groups given by Theorem \ref{existence} above are the first infinite simple sharply multiply transitive groups that do not resemble $\mathrm{PSL}_2(K)$.
\end{rque}



Our proof takes inspiration from the construction given in \cite{RT}, but our approach is more flexible and allows the construction of sharply 2-transitive groups with various properties. The following result illustrates this flexibility. See also Theorem \ref{theo2} below.

\begin{te}[see Theorem \ref{existence2b}]\label{existence2}For every odd integer $n\geq 1$, there exists a sharply 2-transitive group with a simple normal subgroup of index $n$.\end{te}

\begin{rque}
This result is probably true for $n$ even as well, but we only prove it for $n$ odd to avoid technicalities.
\end{rque}

Before stating our next result (which generalizes Theorem \ref{existence}), we need a few definitions. Recall that an \emph{involution} is an element of order 2 in a group. Note that if a group $G$ acts sharply 2-transitively on a set $X$, the stabilizer of any pair of points $(x,y)$ with $x\neq y$ is trivial, and this implies that $G$ contains plenty of involutions: indeed, for any $(x,y)$ with $x\neq y$, there is an element $g\in G$ mapping $(x,y)$ to $(y,x)$, and hence $g^2$ is an involution as it stabilizes $(x,y)$. A sharply $2$-transitive group is said to be of \emph{characteristic 0} if every pair of distinct involutions generates an infinite dihedral group and involutions have fixed points. This definition is motivated by the following observation: if $K$ is a field of characteristic 0, every pair of distinct involutions in $\mathrm{AGL}_1(K)$ generates an infinite dihedral group. Similarly, one can define the notion of a sharply 2-transitive group of characteristic $p$ for any prime number $p\neq 2$ (the characteristic $p=2$ is defined in a different way), and we have $\mathrm{char}(\mathrm{AGL}_1(K))=\mathrm{char}(K)$. See Subsection \ref{Sharp} for further details.


\smallskip

Given a group $G$, we denote by $I_{G}$ the set of involutions of $G$, and by $T_{G}=I_{G}^2\setminus \lbrace 1\rbrace$ the set of non-trivial products of two involutions. An element of $T_{G}$ is called a \emph{translation}. The following result shows that every sharply 2-transitive group of characteristic 0 embeds into a simple sharply 2-transitive group (which is necessarily of characteristic 0 as well).

\begin{te}[see Theorem \ref{theo2b}]\label{theo2}Let $G$ be a sharply 2-transitive group of characteristic 0. There exists a simple sharply 2-transitive group $\Gamma$ that contains $G$, that has the same cardinality as $G$, and such that every element of $\Gamma$ is a translation or a product of two translations (and hence every element is a product of four involutions).\end{te}

\begin{rque}\label{theo3}In fact, for every group $H$ without $2$-torsion, and for every group $G$ as above, there exists a simple sharply 2-transitive group $\Gamma$ that satisfies the conclusions of the theorem, and that has the following property: for every $i\in I_{\Gamma}$, the centralizer $\mathrm{C}_{\Gamma}(i)$ contains a subgroup isomorphic to $H$ (see Theorem \ref{theo3b}). However, there seem to be more restrictions on the centralizer of a translation, and in particular it is still an open problem whether this group has to be abelian (in characteristic 0). We refer the reader to \cite[Section 3]{CT1} for further details.\end{rque}

Our results reinforce the idea, already suggested by \cite{RST} and \cite{RT}, that the class of infinite sharply 2-transitive groups is large, and that there is no hope of classifying these groups without additional assumption. One possible assumption under which one can expect to classify sharply 2-transitive groups comes from model theory: Borovik and Nesin have asked whether a sharply 2-transitive group of \emph{finite Morley rank} is isomorphic to $\mathrm{AGL}_1(K)$ for some algebraically closed field $K$ (see \cite[Question B.61]{BN} and \cite[Chapter 4]{BN} for the definition of a group of finite Morley rank), and a positive answer is conjectured in \cite{ABW} (Conjecture 1). Despite recent progress (see \cite{ABW}, \cite{CT2} and \cite{CT1}), this conjecture remains largely open. If this conjecture is true, then the following one must be true as well.

\begin{conj}
There is no simple sharply 2-transitive group of finite Morley rank.
\end{conj}

\begin{rque}
Note that the conjecture above is also a particular case of the algebraicity conjecture of Cherlin and Zilber, which states that any infinite simple group of finite Morley rank must be an algebraic group over an algebraically closed field.
\end{rque}

Lemma \ref{simple} and \cite[Proposition 7.1]{CT2} show that in a simple sharply 2-transitive group of finite Morley rank, there is a bound $m$ such that every element is a product of at most $m$ translations, which implies immediately that the non-split sharply 2-transitive groups constructed in \cite{RT} do not have finite Morley rank. However, it is not obvious \emph{a priori} that the groups constructed in the present paper do not have finite Morley rank (note in particular that they satisfy the above property with $m=2$), and we therefore ask the following question.

\begin{qu}
Is it true that the simple sharply 2-transitive groups constructed in this paper do not have finite Morley rank?
\end{qu}



Then, it is worth noting that the sharply 2-transitive groups constructed in \cite{RST} and \cite{RT}, as well as in this paper, are not finitely generated. Moreover, the affine group $\mathrm{AGL}_1(K)$ is not finitely generated whenever $K$ is infinite. These observations motivate the following question.

\begin{qu}
Are there infinite finitely generated (simple) sharply 2-transitive groups? Are there finitely presented such groups?
\end{qu}



We conclude this introduction by explaining briefly our construction, which builds on \cite{RT}. First, let us outline the construction of \cite{RT}. It is easy to see that a transitive group action on a set $X$ (with $\vert X\vert\geq 2$) is sharply 2-transitive if and only if the stabilizer $G_x$ of a point $x\in X$ acts regularly (i.e.\ sharply 1-transitively) on $X\setminus \lbrace x\rbrace$, i.e.\ if $G_x$ acts transitively on $X\setminus \lbrace x\rbrace$ and $G_{(x,y)}$ is trivial for any $y\neq x$ (see Lemma \ref{lemme02}). Based on this observation, the construction given in \cite{RT} can be easily explained. The goal is to construct a group $G$ that acts sharply 2-transitively on the set of its involutions, denoted by $I_G$, or equivalently that the following conditions hold:
\begin{enumerate}
\item involutions are conjugate to each other;
\item no non-trivial element centralizes two distinct involutions;
\item for some involution $i\in I_G$, the action of $C_G(i)$ on $I_G\setminus \lbrace i\rbrace$ is transitive.
\end{enumerate}
Start with the sharply 2-transitive group $G_0=\mathrm{AGL}_1(\mathbb{Q})=\mathbb{Q}\rtimes \mathbb{Q}^{\ast}$, and perturb it by taking a free product with $\mathbb{Z}$. Fix two involutions $i\neq j$ in $G_0$. Note that the new group $G_1=\mathrm{AGL}_1(\mathbb{Q})\ast\mathbb{Z}$ is not sharply 2-transitive group, but its involutions are still conjugate to each other, and the centralizers of pairs of distinct involutions remain trivial (conditions (1) and (2) above). The reason why this group is not sharply 2-transitive lies in the fact that the third condition above does not hold, since taking a free product with $\mathbb{Z}$ creates a lot of new involutions and $C_{G_1}(i)$ has no reason to be transitive on the new set of involutions minus $\lbrace i\rbrace$. The idea is now to correct this defect of transitivity: if an involution $k$ is not yet in the orbit of $j$ under the action of $C_{G_1}(i)$, define a new group $G_2=\langle G_1, t \ \vert \ [t,i]=1, \ tjt^{-1}=k\rangle$ (well-defined since both pairs $\lbrace i,j\rbrace$ and $\lbrace i,k\rbrace$ generate an infinite dihedral group). In this new group, $k$ is in the orbit of $j$ under the action of $C_{G_2}(i)$. Then iterate this construction again and again, and let $\Gamma$ be the union of this increasing sequence of groups. By construction, the centralizer of $i$ in $\Gamma$ acts transitively on the set of involutions of $\Gamma$ different from $i$. However, one needs to ensure that centralizers of pairs of distinct involutions remain trivial throughout the construction process, and for this we need to choose the element $k$ very carefully at each step. Here is an example where the centralizer of a pair becomes non-trivial: suppose that the dihedral group $\langle i,k\rangle$ above is strictly contained into another dihedral group $\langle i,k'\rangle$, and that the next group in our sequence is $G_3=\langle G_2, t' \ \vert \ [t',i]=1, \ t'jt'^{-1}=k'\rangle$. Then, there is an element $j'$ in $\langle i,j\rangle$ such that $t'j't'^{-1}=k$, but we know that $j'=zjz^{-1}$ for some $z\in C_{G_0}(i)$ since $C_{G_0}(i)$ acts transitively on $I_{G_0}\setminus \lbrace i\rbrace$, as $G_0$ is sharply 2-transitive. Hence, the non-trivial element $t'zt^{-1}$ centralizes both $i$ and $k$. The following picture illustrates this situation (we think of involutions as points, and we think of a dihedral group generated by two distcint involutions as a line through the corresponding two points; this point of view is useful for understanding the construction).

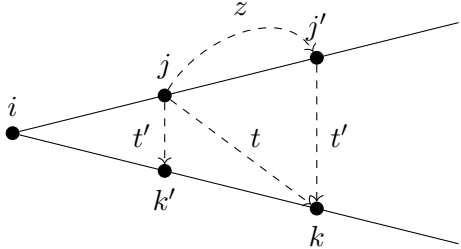
\begin{center}
\begin{tikzpicture}[scale=1]
\node[draw,circle, inner sep=1.7pt, fill, label=below:{$k'$}] (k') at (2,0.5) {};
\node[draw,circle, inner sep=1.7pt, fill, label=above:{$j$}] (j) at (2,1.5) {};
\node[draw,circle, inner sep=1.7pt, fill, label=above:{$i$}] (i) at (0,1) {};
\node[draw,circle, inner sep=1.7pt, fill, label=above:{$j'$}] (j') at (4,2) {};
\node[draw,circle, inner sep=1.7pt, fill, label=below:{$k$}] (k) at (4,0) {};
\node[draw=none, label=above:{$z$}] (1) at (3,2.3) {};
\node[draw=none, label=above:{$t'$}] (2) at (4.3,0.55) {};
\node[draw=none, label=above:{$t'$}] (3) at (1.7,0.55) {};
\node[draw=none, label=above:{$t$}] (4) at (3.2,0.55) {};
\node[draw=none] (5) at (6,2.5) {};
\node[draw=none] (6) at (6,-0.5) {};
\draw[-,>=latex] (i) to (5);
\draw[-,>=latex] (i) to (6);
\draw [->,dashed] (j) to (k);
\draw [->,dashed] (j) to (k');
\draw [->,dashed] (j') to (k);
\draw [->,dashed] (j) to [out=60,in=120] (j');
\end{tikzpicture}
\end{center}

This example shows that $k$ has to be chosen minimal in the sense that $\langle i,k\rangle$ cannot be nested into a larger dihedral group $\langle i,k'\rangle$. The main point of the proof in \cite{RT} is to show that this careful choice is possible at each step (intuitively, this means that the line passing through $i$ and $k$ is discrete). Then, an easy argument shows that the group $\Gamma$ has no non-trivial proper normal abelian subgroup. However, this group is not simple: indeed, Lemma \ref{simple} below claims that a sharply 2-transitive group is simple if and only if it is generated by the set of its translations (i.e.\ the products of two distinct involutions), and it is clear that the new letter added at each step in the construction of \cite{RT} (for instance the letter $t$ above, in the definition of $G_2$) does not belong to the subgroup generated by the translations, and hence $\Gamma$ is not simple. The strategy is therefore clear: one has to ensure that the new letters that appear in the course of the procedure belong to the subgroup generated by the translations. For instance, if one considers the group $G_2$ defined above, a naive idea would be to add a new letter $z$ satisfying the relation $ztz^{-1}$ for some translation $g\in G_1$, but then it is not hard to show that $g$ centralizes two distinct involutions in $G_2$, and thus the construction does not produce a sharply 2-transitive group. The main point of our construction of a simple sharply 2-transitive group is to succeed in implementing the strategy outlined above. For this purpose, we introduce a class of groups (see Definition \ref{class}), denoted by $\mathcal{C}$, by refining the conditions considered in \cite[Theorem 1.1]{RT}. Our class $\mathcal{C}$ is much bigger than the one appearing in \cite{RT} (as evidence, it contains all the sharply 2-transitive groups of characteristic 0) and we show that it is preserved under various HNN extensions and amalgamated free products. Thanks to this robustness, this class $\mathcal{C}$ proved to be a convenient framework in which one can carry out the construction of simple sharply 2-transitive groups.

\section{Preliminaries}

\subsection{Sharply 2-transitive groups}\label{Sharp}






\begin{de}\label{malnormal}A subgroup $H$ of a group $G$ is called \emph{malnormal} if, for every $g\in G\setminus H$, the intersection $gHg^{-1}\cap H$ is trivial. It is called \emph{quasi-malnormal} if, for every $g\in G\setminus H$, the intersection $gHg^{-1}\cap H$ has order at most 2.
\end{de}


The following two lemmas are easy and well-known.

\begin{lemme}[Characterization of malnormality]
A subgroup $H$ in a group $G$ is malnormal if and only if the group $G$ acts on a non-empty set $X$ in such a way that the following conditions hold:
\begin{enumerate}
    \item the action is transitive;
    \item stabilizers of pairs of distinct points are trivial;
    \item $H$ is the stabilizer of a point.
\end{enumerate}
\end{lemme}

\begin{proof}If $H$ is malnormal, then the action of $G$ on the set $G/H$ satisfies the conditions. Conversely, if the conditions hold, set $H=G_x$ and observe that $H\cap gHg^{-1}=G_x\cap G_{g\cdot x}=G_{(x,g\cdot x)}$ is trivial if $g\cdot x \neq x$, and equal to $H$ if $g\cdot x = x$, i.e.\ if $g$ belongs to $G_x=H$.\end{proof}


\begin{lemme}[Characterization of sharply 2-transitivity]\label{lemme02}
Let $G$ be a group acting on a set $X$. Suppose that $X$ has at least two elements. The action is sharply 2-transitive if and only the following three conditions hold:
\begin{enumerate}
        \item the action is transitive;
        \item stabilizers of pairs of distinct points are trivial;
        \item for some (or any) $x\in X$, the action of the stabilizer $G_x$ on $X\setminus \lbrace x\rbrace$ is transitive.
\end{enumerate}
\end{lemme}

\begin{rque}
In view of the previous lemma, the stabilizer $G_x$ is malnormal.\end{rque}


Let $G$ be a group acting sharply 2-transitively on a set $X$ (with $\vert X\vert \geq 2$). Then $G$ has involutions, and all involutions are conjugate. Moreover, if every involution has a fixed point (which is necessarily unique), then all translations are conjugate (recall that a translation is a product of two distinct involutions), and we define the \emph{characteristic} of $G$ as follows: if the order of a translation is infinite, let $\mathrm{char}(G)=0$, and if the order of a translation is finite, we define the characteristic of $G$ as this order. If involutions act freely on $X$, we say that $G$ has characteristic 2. Note that the characteristic is necessarily $0$ or a prime number $p$, and that the case $p = 2$ occurs if and only if involutions are fixed point
free (as otherwise two distinct involutions cannot commute). Last, note that if $\mathrm{char}(G)\neq 2$, then any two distinct involutions are conjugate by an involution (because there is an element swapping their fixed points).

The following result is crucial. It shows that whenever $\mathrm{char}(G)\neq 2$, any sharply 2-transitive action of $G$ on a set $X$ is equivalent to the action of $G$ on $I_G$ by conjugation. More precisely, there is a $G$-equivariant bijection between $X$ and $I_G$. The reader may for instance consult \cite{Tent2}.

\begin{lemme}\label{well_known}Let $G$ be a group acting sharply 2-transitively on a set $X$. Suppose that the characteristic of $G$ is not $2$. Then every involution $i$ of $G$ fixes a unique point $x\in X$, and there is a $G$-equivariant bijection between the set of involutions $I_G$ of $G$ and the set $X$, defined by mapping $i\in I_G$ to its unique fixed point in $X$.\end{lemme}



We will also appeal to the following easy characterisation of simplicity for sharply 2-transitive groups.

\begin{lemme}\label{simple}
A sharply 2-transitive group $G$ of characteristic $\neq 2$ is simple if and only if it generated by the set of translations $T_G$.
\end{lemme}

\begin{proof}
Let $N$ be a normal and nontrivial subgroup of $G$. Let $i\in G$ be an involution, and let $g\in N$ be an element that does not centralize $i$ (such an element exists since every nontrivial element of $G$ centralizes at most one involution). Observe that $igi$ belongs to $N$, and hence $igig^{-1}$ belongs to $N$. Note that $igig^{-1}$ is a translation, and that all translations are conjugate in $G$. As a consequence, $N$ contains all the translations and hence $\langle T_G\rangle\subset N$, which concludes the proof of the proposition.\end{proof}

\begin{rque}This result shows that the groups constructed in \cite{RT} are not simple.\end{rque}


\subsection{Groups acting on trees}\label{trees}

In this subsection, we collect a few well-known results about groups acting on trees.

\begin{de}
A geodesic metric space $(T,d)$ is a \emph{tree} if every geodesic triangle is a tripod. An isometry $\gamma$ of $T$ is called \emph{elliptic} if it fixes a point, and \emph{hyperbolic} otherwise. The tree is called \emph{simplicial} if the infimum of the
distances between vertices is strictly greater than $0$.
\end{de}

The following example will be constantly used in this paper. We refer the reader who is not familiar with Bass-Serre theory to \cite{Serre} or \cite{Bog}.

\begin{ex}
If a group $G$ splits as an amalgamated product or as an HNN extension, then $G$ acts by isometries on the Bass-Serre tree of the splitting, which is a simplicial tree. 
\end{ex}

\begin{prop}
Let $(T,d)$ be a tree, and let $g$ be a hyperbolic isometry of $T$. Then there exists a unique embedded line in $T$, denoted by $A(g)$ and called the \emph{axis} of $g$, on which $g$ acts as translation by $\ell(g)=\inf_{x\in X} d(x,g x)>0$. This real number is called the \emph{translation length} of $g$.
\end{prop}

\begin{prop}\label{dih}
Let $(T,d)$ be a tree, and let $g_1,g_2$ be two elliptic isometries of $T$ with disjoint fixed-point sets. Then $g_1g_2$ is hyperbolic. In particular, if $g_1$ and $g_2$ have order two, then the subgroup $\langle g_1,g_2\rangle$ is an infinite dihedral group.
\end{prop}

\begin{de}Let $(T,d)$ be a tree. We say that two geodesic rays in $T$ starting at a base-point are equivalent if they remain close to each other, and we define the boundary at infinity $\partial_{\infty}T$ of the tree $T$ as the set of equivalence classes of geodesic rays starting at a base-point.\end{de}

If a group $G$ acts on $T$ by isometries, this action extends to an action on $\partial_{\infty}T$ that turns out to be an extremely useful tool for studying the group $G$. In particular, we will need the following result.

\begin{prop}
Let $(T,d)$ be a tree, let $G$ be a group acting on $T$ by isometries, and let $g\in G$ be a hyperbolic element. Then $g$ fixes exactly two distinct points on $\partial_{\infty}T$ denoted by $g_{-}$ and $g_{+}$. The stabilizer in $G$ of the pair $\lbrace g_{-}, g_{+}\rbrace$ coincides with $\mathrm{Stab}(A(g))$, the setwise stabilizer of the axis of $g$. Moreover, if $T$ is simplicial, then $\mathrm{Stab}(A(g))$ has a clear description:
\begin{itemize}
    \item if there is no element in $G$ swapping $g_{-}$ and $g_{+}$, then $\mathrm{Stab}(A(g))$ is a semi-direct product $S\rtimes \langle g'\rangle$ where $S\subset G$ denotes the set of elements fixing the axis $A(g)$ pointwise and $g'\in G$ is a hyperbolic element such that $A(g')=A(g)$ and $g'$ has minimal translation length.
    \item if there is an element in $G$ swapping $g_{-}$ and $g_{+}$, then $\mathrm{Stab}(A(g))$ has a subgroup $\mathrm{Stab}^{+}(A(g))$ of index two that fits into the first point above.
\end{itemize}
\end{prop}

\section{A class of groups}\label{section_class}

In this section, we introduce a class of groups $\mathcal{C}$ by refining the conditions from Theorem 1.1 in \cite{RT}. Our class $\mathcal{C}$ contains all the sharply 2-transitive groups in characteristic $0$ (see Lemma \ref{contient_tout} below), for instance $\mathrm{AGL}_1(K)=K\rtimes K^{*}$ for any field $K$ of characteristic $0$, and we will prove that $\mathcal{C}$ is remarkably stable under a variety of algebraic operations. This stability will allow us to construct sharply 2-transitive groups with unexpected properties inside the class $\mathcal{C}$.

Let $i$ be an involution (in a group $G$) that will be fixed throughout this paper. If $j,k$ are two involutions distinct from $i$, we say that $j$ and $k$ are equivalent, denoted $j\sim k$, if there exist $m,n\in\mathbb{Z}\setminus\lbrace 0\rbrace$ such that $(ij)^m=(ik)^n$. We say that $j$ is \emph{minimal} if for any involution $k$, if $\langle i,j\rangle\subset \langle i,k\rangle$ then $\langle i,j\rangle = \langle i,k\rangle$. We denote by $A$ the centralizer $C_G(i)$ of $i$ in $G$, and we denote by $\mathrm{Orb}_A(j)$ the orbit of $j$ under the action of $A$ by conjugation.

\begin{de}\label{class}We say that a group $G$ belongs to the class $\mathcal{C}$ if there exist three pairwise distinct involutions $i,j,s\in G$ such that $j=sis$, and if the following conditions hold in $G$, where $A$ denotes the centralizer of the involution $i$.
\begin{enumerate}
    \item All involutions of $G$ are conjugate.
    \item Any two distinct involutions of $G$ generate an infinite dihedral group.
    \item For any $k\in I$, if $k\sim j$ then $k\in \mathrm{Orb}_A(j)$.
    \item $A$ is malnormal in $G$.
    \item For any $k\in I\setminus \lbrace i\rbrace$ such that $k\notin \mathrm{Orb}_A(j)$, there is an involution $k'\in I$ such that the following hold:
    \begin{itemize}
        \item $\langle i,k\rangle\subset \langle i,k'\rangle$ (in particular, $k\sim k'$);
        \item $k'$ is minimal;
        \item $\langle i,k'\rangle$ is quasi-malnormal.
    \end{itemize}
\end{enumerate}
\end{de}

\begin{rque}Our class $\mathcal{C}$ contains all the sharply 2-transitive groups of characteristic 0 (see Lemma \ref{contient_tout} below), and in particular $\mathcal{C}$ contains all groups considered in \cite{RT}. Let us point out the differences between the class $\mathcal{C}$ and the six conditions considered in \cite[Theorem 1.1]{RT}. The fifth condition in \cite{RT} is a strong assumption which implies notably that the centralizer of a translation is abelian and composed only of translations. It is an open problem whether this condition must be satisfied by all sharply 2-transitive groups in characteristic $0$. We replace this condition by the (\emph{a priori} much weaker) assumption that $A$ is malnormal (condition \emph{(4)} above). In addition, we slightly modify the last condition from \cite[Theorem 1.1]{RT} and assume that $\langle i,k'\rangle$ is quasi-malnormal. Last, note that we do not assume that the third condition from \cite[Theorem 1.1]{RT} holds in our class; in fact, this condition follows from the other conditions, as pointed out by Felix Pott in his master's thesis (see Lemma \ref{lemme0} below).
\end{rque}

\begin{rque}Note that conditions \emph{(1)} and \emph{(2)} are preserved under any HNN extension (see Lemma \ref{pres_HNN}).\end{rque}

\begin{lemme}\label{contient_tout}
The class $\mathcal{C}$ contains all the sharply 2-transitive groups of characteristic 0.
\end{lemme}

\begin{proof}Condition \emph{(2)} is true by definition in any sharply 2-transitive group of characteristic 0. The malnormality of $A$ follows from Lemma \ref{lemme02} (and the remark below this lemma), together with Lemma \ref{well_known}.
The fact that involutions are conjugate, as well as the existence of the three pairwise distinct involutions $i,j,s$ such that $sjs=j$, are well-known (see Subsection \ref{Sharp}). Last, note that conditions \emph{(3)} and \emph{(5)} hold in any sharply 2-transitive group of characteristic 0 since every involution different from $i$ is in the orbit of $j$ under $A$ (again by Lemmas \ref{lemme02} and Lemma \ref{well_known}), and hence condition \emph{(3)} is automatic and condition \emph{(5)} is vacuous.\end{proof}

It is worth noting that, as a first step in the construction of simple sharply 2-transitive groups of characteristic 0, we will prove that every group $G\in \mathcal{C}$ is contained in a sharply 2-transitive group of characteristic 0 (see Theorem \ref{theo4}).

\subsection{Preliminary lemmas}

In this subsection, we prove some easy results about the class $\mathcal{C}$ that will be useful in the rest of the paper.

\begin{lemme}\label{pres_HNN}Conditions \emph{(1)} and \emph{(2)} from Definition \ref{class} are preserved under any HNN extension.\end{lemme}

\begin{proof}Let $G$ be a group in the class $\mathcal{C}$, and let $G_1$ be an HNN extension of $G$. Let us prove that $G_1$ satisfies the conditions \emph{(1)} and \emph{(2)}.

Condition \emph{(1)}: let $k\in G_1$ be an involution. Since $k$ has finite order, it is elliptic in the Bass-Serre tree of the splitting of $G_1$ as an HNN extension. Therefore, there is an element $g\in G_1$ such that $k$ belongs to $gGg^{-1}$. It follows that $g^{-1}kg$ is conjugate to $i$, and thus $k$ is conjugate to $i$.

Condition \emph{(2)}: let $k,\ell\in G_1$ be two involutions. As above, $k$ and $\ell$ are elliptic in the Bass-Serre tree. If they fix a common vertex, then there exists $g\in G_1$ such that $k$ and $\ell$ belong to $gGg^{-1}$. Since $G$ satisfies condition \emph{(2)}, we see that $g^{-1}kg$ and $g^{-1}\ell g$ generate an infinite dihedral group, and thus $k$ and $\ell$ generate an infinite dihedral group. Then, if $k$ and $\ell$ do not fix a common vertex in the Bass-Serre tree, Lemma \ref{dih} tells us that they generate an infinite dihedral group.
\end{proof}

\begin{lemme}\label{pres_amalgam}Let $G,G'$ be two groups in $\mathcal{C}$. Let $i$ and $i'$ be two involutions in $G$ and $G'$ respectively. Then the amalgamated free product $G_1=G\ast_{i=i'}G'$ satisfies the conditions \emph{(1)} and \emph{(2)}.\end{lemme}

\begin{proof}Condition \emph{(1)}: any involution in $G_1$ is elliptic in the Bass-Serre tree of the splitting of $G_1$ as an amalgamated free product, which means that it belongs to a conjugate of $G$ or $G'$. Now, observe that all involutions in $G$ are conjugate to $i$, all involutions in $G'$ are conjugate to $i'$, and $i=i'$. 

Condition \emph{(2)}: let $k,\ell$ be two involutions in $G_1$. They are elliptic in the Bass-Serre tree. If they belong to the same vertex group, they generate an infinite dihedral group by condition \emph{(2)} in $G$ and $G'$. If they fix different vertices, they also generate an infinite dihedral group by Lemma \ref{dih}.
\end{proof}

The following lemma can be proved in a similar way.

\begin{lemme}\label{pres_amalgam2}Let $G$ be a group in $\mathcal{C}$, and let $G'$ be a group. Let $i$ and $i'$ be two involutions in $G$ and $G'$ respectively. Suppose that $i'$ is the only involution in $G'$. Then the amalgamated free product $G_1=G\ast_{i=i'}G'$ satisfies the conditions \emph{(1)} and \emph{(2)}.\end{lemme}

\begin{lemme}\label{lemma1}Let $G$ be a group in $\mathcal{C}$. The following two assertions are equivalent.
\begin{enumerate}
    \item[(i)] $A$ is malnormal in $G$.
    \item[(ii)] For every involution $k\neq i$, one has: \[A\cap C_{G}(k)=\lbrace 1\rbrace.\]In other words, the centralizer of the pair $(i,k)$ is trivial.
\end{enumerate}
\end{lemme}

\begin{proof}Suppose that \emph{(ii)} holds, and prove that $A=C_G(i)$ is malnormal in $G$. Let $g\in G\setminus A$ and let $h\in gAg^{-1}\cap A$. Since $gAg^{-1}=C_{G}(gjg^{-1})$, the element $h$ centralizes both $k=gig^{-1}$ and $i$. Moreover, note that $k$ is different from $i$ since $g$ does not centralize $i$. Hence $h=1$.

Conversely, suppose that $A$ is malnormal in $G$, and prove that the assertion \emph{(ii)} holds. Let $k\neq i$ be an involution, and let $h\in A\cap C_{G}(k)$. Note that $h=khk^{-1}$ belongs to $A\cap kAk^{-1}$, which is trivial since $A$ is malnormal and $k$ does not belong to $A$ (indeed, $k\neq i$ and $i$ is the only involution in $A$, because two distinct involutions generate an infinite dihedral group).\end{proof}

\begin{lemme}\label{lemme0}Let $G$ be a group in $\mathcal{C}$. Then for any $n > 0$ there exists a unique involution $k\in G$ such that $(ik)^n=ij$.
\end{lemme}

\begin{proof}Observe that $i(ij)^n\sim j$. By the third condition in the definition of the class $\mathcal{C}$, there is an element $a\in A$ such that $i(ij)^n=aja^{-1}$, and therefore $(ik)^n=ij$ with $k=aja^{-1}$. To see that $k$ is unique, observe that for any involution $k'$ such that $(ik)^n=(ik')^n=ij$, there is an element $a\in A$ such that $k'=aka^{-1}$. It follows that $a i(ik)^n a^{-1}=i(ik)^n$. Hence $a$ centralizes two distinct involutions, namely $i(ik)^n$ and $i$, and thus $a=1$ (by Lemma \ref{lemma1} above) and $k=k'$.
\end{proof}

\begin{lemme}\label{lemme_translation}Let $G$ be a group in $\mathcal{C}$. Then $A$ contains no translation.
\end{lemme}

\begin{proof}Suppose towards a contradiction that there exists a translation $t=ab$ in $A$, where $a,b$ denote two distinct involutions. First, observe that $b$ and $a$ are both different from $i$. Indeed, if $b=i$ then $tit=aia=i$, which is not possible since the group generated by $i$ and $a$ (with $a\neq b=i$) is an infinite dihedral group by condition \emph{(2)} in the definition of the class $\mathcal{C}$. Hence $b$ and $i$ are different, and similarly $a$ and $i$ are different. Then, define $k=bib$. Note that $k$ is an involution, and that $k$ is different from $i$ since $i$ and $b$ generate an infinite dihedral group (as $i\neq b$). Since $b$ has order 2, it swaps $i$ and $k$. But $aka=abiba=tit=i$ as $t$ belongs to $A$, and thus $a$ swaps $i$ and $k$ as well. It follows that $t$ centralizes both $i$ and $k$. But Lemma \ref{lemma1} tells us that no nontrivial element can centralize the pair $(i,k)$. Therefore $t=1$, contradicting our assumption.\end{proof}

\section{Stability of $\mathcal{C}$ under HNN extensions and amalgamations}

We shall prove that $\mathcal{C}$ is stable under certain HNN extensions and amalgamated free products. For the convenience of the reader, the results are stated here, and the proofs are given in Section \ref{proofs}.

\subsection{Stability under HNN extensions}~\ \smallskip

The following proposition is the analogue of Proposition 1.4 in \cite{RT}.

\begin{prop}[see Proposition \ref{jesaispas0b}]\label{jesaispas0}
Let $G$ be a group in $\mathcal{C}$. Let $k\in G$ be an involution distinct from $i$. Suppose that $k$ does not belong to $\mathrm{Orb}_A(j)$ and that $k$ is minimal. Define \[G_1=\langle G,t \ \vert \ tit^{-1}=i, \ tjt^{-1}=k\rangle\] Then $G_1$ belongs to $\mathcal{C}$.
\end{prop}

The following proposition will be used to ensure that all elements in our group are products of (at most two) translations, which is crucial to establish the simplicity of the group, in view of Lemma \ref{simple}.

\begin{prop}[see Proposition \ref{jesaispasb}]\label{jesaispas}
Let $G$ be a group in $\mathcal{C}$. Let $g,h\in G$ be two elements of order $n>2$. Suppose that the following conditions hold:
\begin{enumerate}
\item $\langle h\rangle$ contains no translation;
\item for every $y\in G$, $y\langle h\rangle y^{-1}\cap \langle g\rangle =\lbrace 1\rbrace$;
\item $\langle h\rangle$ is malnormal (i.e.\ $N_G(\langle h\rangle)=\langle h\rangle$).
\end{enumerate}
Define $G_1=\langle G,z \ \vert \ zgz^{-1}=h\rangle$. Then $G_1$ belongs to $\mathcal{C}$.
\end{prop}

\begin{rque}
Note that there are two basic obstructions for a group of the form $G_1=\langle G,z \ \vert \ zgz^{-1}=h\rangle$ to belong to the class $\mathcal{C}$, and this proposition shows that these are essentially the only obstructions.
\begin{enumerate}
    \item \emph{First obstruction:} if $g$ centralizes an involution and $h$ is a translation, then $G_1$ contains a translation that centralizes an involution, and hence $G_1$ cannot be in $\mathcal{C}$ by virtue of Lemma \ref{lemme_translation} which asserts that the centralizer of an involution cannot contain a translation. This obstruction is ruled out by the first assumption in the proposition above, namely that $\langle h\rangle$ contains no translation.
    \item \emph{Second obstruction:} if $g$ centralizes an involution and $h$ centralizes an involution, then $g$ centralizes two distinct involutions in $G_1$, and this is not possible according to Lemma \ref{lemma1} which says that no nontrivial element centralizes a pair of distinct involutions in a group belonging to the class $\mathcal{C}$. This second obstruction is ruled out by the third assumption in the proposition above, namely that $\langle h\rangle$ is malnormal (which implies that $h$ does not commute with any involution).
\end{enumerate}
\end{rque}

\subsection{Stability under amalgamated products}~\ \smallskip

Let $K$ be a field. Observe that in the affine group $\mathrm{AGL}_1(K)$, involutions do not belong to the subgroup generated by the translations. The goal of the following proposition is to embed $\mathrm{AGL}_1(K)$ (or more generally any group from the class $\mathcal{C}$) into a group that belongs to $\mathcal{C}$ and in which some (equivalently, any) involution is a product of two translations.

\begin{prop}[see Corollary \ref{coro10}]\label{coro1}
Let $G$ be a group in $\mathcal{C}$. Consider the following group $H$: \[H=\Biggl\langle 
       \begin{array}{l|cl}
                        & abcde=1 \\
            a,b,c,d,e,u,v,w,x & a^2=b^2=c^2=d^2=e^2=1 \\
                        & ubu^{-1}=vcv^{-1}=wdw^{-1}=xex^{-1}=a  &                                             
        \end{array}
     \Biggr\rangle.\]
Then $G_1=G\ast_{i=a}H$ belongs to $\mathcal{C}$.
\end{prop}

\begin{rque}
Note that the involution $i=a$ is a product of two translations in $G_1$, namely $bc$ and $de$. The additional letters $u,v,w,x$ in the presentation of $H$ are here to ensure that there is only one conjugacy class of involutions in $H$, hence in $G_1$.
\end{rque}

The following result is similar in spirit to the previous one. It will be used in combination with Proposition \ref{jesaispas} to construct, given a group $G\in \mathcal{C}$ and an element $g\in G$ of order $>2$, a group $G'\in \mathcal{C}$ such that $G\subset G'$ and $g$ is a product of two translations in $G'$.

\begin{prop}[see Corollary \ref{coro20}]\label{coro2}
Let $G$ be a group in $\mathcal{C}$. Let us define the following groups, for $n\geq 2$:\[H_{\infty}=\Biggl\langle 
       \begin{array}{l|cl}
                        & abcdh=1 \\
            a,b,c,d,h,u,v,w & a^2=b^2=c^2=d^2=1 \\
                        & ubu^{-1}=vcv^{-1}=wdw^{-1}=a  &                                             
        \end{array}
     \Biggr\rangle,\]
     \[H_{2n-1}=\Biggl\langle 
       \begin{array}{l|cl}
                        & abcdh=1 \\
            a,b,c,d,h,u,v,w & a^2=b^2=c^2=d^2=h^{2n-1}=1 \\
                        & ubu^{-1}=vcv^{-1}=wdw^{-1}=a  &                                             
        \end{array}
     \Biggr\rangle,\]
     \[H_{2n}=\Biggl\langle 
       \begin{array}{l|cl}
                        & abcdh=1 \\
            a,b,c,d,h,u,v,w,x & a^2=b^2=c^2=d^2=h^{2n}=1 \\
                        & ubu^{-1}=vcv^{-1}=wdw^{-1}=xh^nx^{-1}=a  &                                             
        \end{array}
     \Biggr\rangle.\]
Let $H_m$ denote one of the previous groups (with $m\geq 3$, possibly infinite). Then $G_1=G\ast_{i=a}H_m$ belongs to $\mathcal{C}$. Moreover, the following conditions hold, for any element $g\in G$ of order $m$ (these conditions should be compared with the assumptions of Proposition \ref{jesaispas}):
\begin{enumerate}
    \item $\langle h\rangle$ contains no translation;
    \item for every $y\in G_1$, $y\langle g\rangle y^{-1}\cap \langle h\rangle=\lbrace 1\rbrace$;
    \item $\langle h\rangle$ is malnormal in $G_1$.
\end{enumerate}
\end{prop}

\begin{rque}
The distinction between the presentations of $H_{2n-1}$ and $H_{2n}$ lies in the fact that, in the latter case, the cyclic group $\langle h\rangle$ contains an involution, which has to be conjugate to $a$.\end{rque}

The following two propositions can be proved in a similar manner as the previous results. Proposition \ref{i} below shows that one can extend the centralizer of the involution $i$ by a group without 2-torsion, and Proposition \ref{free} shows that one can perform a free product by a group without 2-torsion.


\begin{prop}\label{i}
Let $G$ be a group in $\mathcal{C}$. Let $H$ be a group containing no $2$-torsion. Define $G_1=G\ast_i(\langle i\rangle \times H)$. Then $G_1$ belongs to $\mathcal{C}$.
\end{prop}

\begin{prop}\label{free}
Let $G$ be a group in $\mathcal{C}$. Let $H$ be a group containing no $2$-torsion. Define $G_1=G\ast H$. Then $G_1$ belongs to $\mathcal{C}$.
\end{prop}

\subsection{Stability under increasing sequences}

Last, let us mention an important fact, whose proof is straightforward.

\begin{prop}
Let $(G_i)_{i\in I}$ be an increasing sequence of groups in $\mathcal{C}$. Then the union of the sequence belongs to $\mathcal{C}$.
\end{prop}

In the next section, we shall prove our main results, namely the construction of simple sharply 2-transitive groups with various properties.

\section{Simple sharply 2-transitive groups}

The following theorem is the first step of the construction.

\begin{te}\label{theo4}
For every group $G\in \mathcal{C}$, there exists a sharply 2-transitive group $\Gamma$ that contains $G$.\end{te}

\begin{proof}If $G$ is already sharply 2-transitive, there is nothing to do. From now on, assume that $G$ is not a sharply 2-transitive group. Set $G_0=G$ and $A_0=A$. Let $k_1,k_2,k_3,\ldots$ be an enumeration of the minimal involutions of $G$ (as defined in the second paragraph of Section \ref{section_class}) that are not in the orbit of $j$ under the action of $A$ by conjugation, and suppose that they are pairwise non equivalent (which means that $ig_n$ and $ig_m$ have no non-trivial common powers if $n\neq m$). Note that if $G$ is uncountable, we cannot index our elements by natural numbers, but the proof is exactly the same and we assume that $G$ is countable for simplicity. Define $G_0^1\subset G_0^1\subset \cdots$ by using Proposition \ref{jesaispas0} repeatedly, that is \[G_0^1=\langle G_0, t_1 \ \vert \ t_1it_1^{-1}=i, \ t_1jt_1^{-1}=k_1\rangle,\]
\[G_0^2=\langle G_0^1, t_2 \ \vert \ t_2it_2^{-1}=i, \ t_2jt_2^{-1}=k_2\rangle,\]and so on. Let $G_1$ be the union of this increasing sequence of groups, and let $A_1$ be the stabilizer of $i$ in this group. Then, iterate this construction to get an increasing sequence $G_0\subset G_1\subset G_2\subset \cdots$ and let $\Gamma$ be the union. We also define $A_{\Gamma}$ to be the centralizer of $i$ in $\Gamma$. We shall prove that $\Gamma$ is a sharply 2-transitive group by applying Lemma \ref{lemme02} to the action of $\Gamma$ by conjugation on the set $I_{\Gamma}$ of its involutions. Recall that we have to prove that the following three conditions hold:
\begin{enumerate}
    \item the action is transitive (i.e.\ $I_{\Gamma}$ is a conjugacy class);
    \item stabilizers of pairs of distinct points are trivial, i.e.\ every non-trivial element $\gamma\in\Gamma$ centralizes at most one involution;
    \item $A_{\Gamma}=C_{\Gamma}(i)$ acts transitively on $I_{\Gamma}\setminus \lbrace i\rbrace$.
\end{enumerate}

First, it is clear that involutions are conjugate to each other in $\Gamma$ since they are conjugate in every $G_n$ (condition \emph{(1)} in the definition of $\mathcal{C}$).

Then, let $\gamma\in\Gamma $ be an element that centralizes $i$ and $k\neq i$. There is an integer $n$ such that $i,k,\gamma$ belong to $G_n$. Note that $i$ is the only involution in $A_n=C_{G_n}(i)$ since two distinct involutions in $G_n$ generate an infinite dihedral group, hence $k$ does not belong to $A_n$. By Lemma \ref{lemma1}, since $A_n$ is malnormal in $G_n$, $\gamma$ is trivial.

Last, note that for every integer $n$, the centralizer $A_{n+1}$ of $i$ in $G_{n+1}$ acts transitively on the set of involutions contained in $G_n\setminus \lbrace i\rbrace$: indeed, if an involution $k\in G_n\setminus \lbrace i \rbrace$ is not in the orbit of $j$ under $A_n$, then there is a minimal involution $k'$ such that $\langle i,k\rangle \subset \langle i,k'\rangle$ (and in particular $k$ and $k'$ are equivalent), and by construction $k'$ is in the orbit of $j$ under $A_{n+1}$, i.e.\ there is an element $a\in A_{n+1}$ such that $ak'a^{-1}=j$. In particular, $ak'a^{-1}$ is equivalent to $j$. As $k$ is equivalent to $k'$, by transitivity $aka^{-1}$ is equivalent to $j$. Hence, by the third condition in the definition of the class $\mathcal{C}$, there exists an element $a'\in A_{n+1}$ such that $(a'a)k(a'a)^{-1}=j$. It follows that $A_{\Gamma}=C_{\Gamma}(i)$ acts transitively on $I_{\Gamma}$ minus $i$. 

Hence, $\Gamma$ acts sharply 2-transitively on $I_{\Gamma}$.\end{proof}

We are ready to prove Theorem \ref{theo2}.

\begin{te}\label{theo2b}Let $G$ be a sharply 2-transitive group of characteristic 0. There exists a simple sharply 2-transitive group $\Gamma$ that contains $G$, that has the same cardinality as $G$, and such that every element of $\Gamma$ is a translation or a product of two translations (and hence every element is a product of four involutions). \end{te}

\begin{proof}Consider $G_1=G\ast \mathbb{Z}$. This group is in the class $\mathcal{C}$ by Proposition \ref{free}. Let $T_1$ denote the set of translations of $G_1$, and let $g_1,g_2,g_3,\ldots$ be an enumeration of the elements of $G_1$ that are not in $T_1^2$. As a first step, we will construct a sharply 2-transitive group $G_1^1$ that contains $G_1$ and such that $g_1$ is a product of two translations in $G_1^1$. We distinguish two cases.

\emph{First case}. If $g_1$ has order $2$, consider the group $G_1\ast_{i=a}H$ given by Corollary \ref{coro1}. In this group, $i=a$ is a product of two translations by definition of $H$. Moreover, this group belongs to the class $\mathcal{C}$ by Corollary \ref{coro1}. Hence, once can apply Theorem \ref{theo4} above to get a sharply 2-transitive group $G_1^1$ that contains $G_1\ast_{i=a}H$ (and thus $G_1$) and such that $g_1$ is a product of two translations in $G_1^1$.

\emph{Second case}. If $g_1$ has order $n>2$, consider the group $G_1\ast_{i=a}H_n$ given by Corollary \ref{coro2}. In this group, we have a new element $h$ of order $n$ which is a product of two translations by definition of $H_n$, and such that $g_1$ and $h$ satisfy the assumptions of Proposition \ref{jesaispas}. Since the group $G_1\ast_{i=a}H_n$ lies in the class $\mathcal{C}$ by Corollary \ref{coro2}, one can apply Proposition \ref{jesaispas} in order to get an HNN extension of $G_1\ast_{i=a}H_n$ in which $g_1$ and $h$ are conjugate, and that lies in $\mathcal{C}$. Note that in this HNN extension, the element $g_1$ is now a product of two translations. Hence, once can apply Theorem \ref{theo4} to get a sharply 2-transitive group $G_1^1$ that contains $G_1\ast_{i=a}H$ (and thus $G_1$) and such that $g_1$ is a product of two translations in $G_1^1$.

Then, take $G_1^1\ast\mathbb{Z}$ and construct a sharply-2-transitive group $G_1^2$ in which $g_1,g_2$ are products of two translations, and so on. We get an increasing sequence $G_1^1\subset G_1^2\subset \cdots$ of sharply 2-transitive groups. Let $G_2$ be the union of this increasing sequence. Observe that this group is sharply 2-transitive as an increasing union of sharply 2-transitive groups of characteristic 0, and that every element of $G_1$ can be written as a product of two translations in $G_2$.

We iterate this construction and get an increasing sequence $G_1\subset G_2\subset \cdots$. Let $\Gamma$ be the union of this sequence. This group is sharply 2-transitive and every element of $\Gamma$ is a translation or a product of two translations in $\Gamma$.

Last, note that $\Gamma$ is simple, as an immediate consequence of Proposition \ref{simple}.
\end{proof}

Then, the construction of the group $\Gamma$ mentioned in Remark \ref{theo3} is an easy adaptation of the previous proof, see below.

\begin{te}\label{theo3b}Let $G$ be a sharply 2-transitive group of characteristic 0. Let $H$ be a group without $2$-torsion. Then there exists a simple sharply 2-transitive group $\Gamma$ such that the following three conditions hold:
\begin{enumerate}
    \item $\Gamma$ contains $G$;
    \item every element of $\Gamma$ is a translation or a product of two translations, and hence every element is a product of four involutions;
    \item for every $i\in I_{\Gamma}$, $\mathrm{C}_{\Gamma}(i)$ contains a subgroup isomorphic to $H$.
\end{enumerate}
\end{te}

\begin{proof}
Consider $G_1=(G\ast \mathbb{Z})\ast_{\langle i\rangle}(\langle i\rangle\times H)$. This group is in the class $\mathcal{C}$ by Propositions \ref{free} and \ref{i}. Then the proof is exactly the same as above.
\end{proof}

Last, we prove Theorem \ref{existence2}.

\begin{te}\label{existence2b}For every odd integer $n\geq 1$, there exists a sharply 2-transitive group with a simple normal subgroup of index $n$.\end{te}

\begin{proof}Let $\zeta_n$ be an $n$th root of unity, and let $K=\mathbb{Q}(\zeta_n)$. Consider the groups $G_0=\mathrm{AGL}_1(K)$ and $G_1=G_0\ast\mathbb{Z}$. We proceed as in the proofs of the previous theorems, but we make the following changes:
\begin{itemize}
    \item we add new letters throughout the construction in order to ensure that the cyclic subgroups of order $n$ are conjugate to each other;
    \item we do not conjugate the elements of the cyclic subgroups of order $n$ to products of translations.
\end{itemize}
As in the previous proofs, let $\Gamma$ be the group obtained at the end of the construction, and define $N=\langle T_{\Gamma}\rangle$, the subgroup generated by the translations. This subgroup is normal since $T_{\Gamma}$ is a conjugacy class. Then, observe that killing $N$ gives rise to a finite quotient of order $n$: by construction, every element that does not belong to a cyclic subgroup of order $n$ is killed in $\Gamma/N$ as it is a product of two translations. In particular, the letters that conjugate the cyclic subgroups of order $n$ to each other are killed, and thus all these groups are identified in the quotient. In addition, no element of order $n$ is killed in the quotient, by construction.\end{proof}

\begin{rque}
If we slightly modify the previous proof and do not conjugate the involution $i$ to a product of translations, one gets a quotient of order $2n$.
\end{rque}

\section{Proof of Propositions \ref{jesaispas0} to \ref{coro2}}\label{proofs}

We shall prove that $\mathcal{C}$ is stable under certain HNN extensions and amalgamated free products.

\subsection{Stability of $\mathcal{C}$ under HNN extensions}

We begin by proving two lemmas.

\begin{lemme}\label{lemme}
Let $G$ be a group in $\mathcal{C}$. Let $H,K$ be two isomorphic subgroups of $G$, let $\alpha:H\rightarrow K$ be an isomorphism, and let $G_1=G_{\alpha}=\langle G,t \ \vert \ tht^{-1}=\alpha(h), \ \forall h\in H\rangle$. Suppose that $K$ is quasi-malnormal in $G_1$, and that for every $g\in G$, $\vert gHg^{-1}\cap K\vert \leq 2$.
Then, for $x\in G$, the following assertions hold.
\begin{itemize}
\item If $x$ is not contained in a conjugate of $H$ or $K$, then $\lbrace g\in G_1 \ \vert \ gxg^{-1}\in G\rbrace$ is contained in $G$.
\item If $x$ belongs to $H$ and $x$ has order $>2$, then $\lbrace g\in G_1 \ \vert \ gxg^{-1}\in H\rbrace$ is contained in $G$. 
\item If $x$ belongs to $K$ and $x$ has order $>2$, then $\lbrace g\in G_1 \ \vert \ gxg^{-1}\in K\rbrace$ is contained in $tGt^{-1}$. 
\end{itemize}
In all the above cases, the normalizer of $\langle x\rangle$ in $G_1$ and the centralizer of $x$ in $G_1$ are elliptic in the Bass-Serre tree of the splitting of $G_1$. Moreover, if $x$ does not belong to a conjugate of $K$ and $x$ has order $>2$, then the normalizer of $\langle x\rangle$ in $G_1$ and the centralizer of $x$ in $G_1$ are contained in $G$. 
\end{lemme}

\begin{proof}
\emph{First case.} The claim can be proved easily by means of reduced normal forms. Alternatively, let us give a geometrical argument. If $x$ is not contained in a conjugate of $H$ or $K$ by an element of $G$, then consider the Bass-Serre tree corresponding to the HNN extension, and let $v$ be the unique vertex fixed by $G$. Since $x$ is not contained in an edge group, $v$ is the unique vertex fixed by $x$. Now, if $g\in G_1$ is such that $gxg^{-1}$ lies in $G$, then $gxg^{-1}\cdot v= v$ and thus $x \cdot (g^{-1}\cdot v)=g^{-1}\cdot v$. By uniqueness of the fixed point of $x$, one gets $g^{-1}\cdot v =v$. Hence $g$ belongs to $G$. 

\emph{Second case.} Suppose that $x$ belongs to $H$. Let $g\in G_1$ be an element such that $gxg^{-1}\in H$. Suppose towards a contradiction that $g\notin G$ and consider a reduced form $g=g_0t^{\varepsilon_1}g_1\cdots t^{\varepsilon_{n+1}}g_{n+1}$, with $n\geq 1$, $g_i\in G$, $\varepsilon_i=\pm 1$, and for $1\leq i\leq n$, if $\varepsilon_i=-\varepsilon_{i+1}=1$ then $g_i\notin H$ and if $\varepsilon_i=-\varepsilon_{i+1}=-1$ then $g_i\notin K$. Since $gxg^{-1}$ belongs to $H$ by assumption, a simplification occurs necessarily at $t^{\varepsilon_{n+1}}g_{n+1}xg_{n+1}^{-1}t^{-\varepsilon_{n+1}}$. Recall that $x$ belongs to $H$, has order at least 3, and that $\vert g_{n+1}Hg_{n+1}^{-1}\cap K\vert \leq 2$. As a consequence, $g_{n+1}xg_{n+1}^{-1}$ is in $H$, $\varepsilon_{n+1}=1$ and $t^{\varepsilon_{n+1}}g_{n+1}xg_{n+1}^{-1}t^{-\varepsilon_{n+1}}$ belongs to $K$. The same argument shows that $g_nt^{\varepsilon_{n+1}}g_{n+1}xg_{n+1}^{-1}t^{-\varepsilon_{n+1}}g_n^{-1}$ belongs to $K$ and that $\varepsilon_{n}=-1$. Moreover, since $K$ is quasi-malnormal, $g_n$ lies in $K$. Hence the reduced form of $g$ contains the sequence $t^{-1}g_nt$ with $g\in K$, which contradicts the fact that the form is reduced. Hence $n=1$ and $g=g_0tg_1$. One gets $tg_1xg_1^{-1}t^{-1}\in g_0^{-1}Hg_0$, which is not possible since $g_1xg_1^{-1}$ belongs to $H$, $tg_1xg_1^{-1}t^{-1}$ belongs to $K$, and $g_0^{-1}Hg_0\cap K$ has order at most 2.

\emph{Third case.} Suppose that $x$ belongs to $K$. Then the result follows immediately from the second case since $K=tHt^{-1}$.\end{proof}

\begin{lemme}\label{lemme2}Let $G$ be a group in $\mathcal{C}$. Let $H,K$ be two isomorphic subgroups of $G$, let $\alpha:H\rightarrow K$ be an isomorphism, and let $G_1=G_{\alpha}=\langle G,t \ \vert \ tht^{-1}=\alpha(h), \ \forall h\in H\rangle$. Suppose that $K$ is quasi-malnormal in $G_1$, and that for every $g\in G$, $\vert gHg^{-1}\cap K\vert \leq 2$. Moreover, suppose that $H$ and $K$ are isomorphic to $\mathbb{Z}/n\mathbb{Z}$ for $n>2$, or to $\mathbb{Z}$, or to the infinite dihedral group $D_{\infty}$. Let $k\in G_1$ be an involution. If $ik$ acts hyperbolically on the Bass-Serre tree of the splitting of $G_1$, then there exists an involution $k'\neq i$ in $G_1$ such that the following three conditions are satisfied (compare with condition \emph{(5)} in the definition of the class $\mathcal{C}$):
\begin{enumerate}
        \item[(a)] $\langle i,k\rangle\subset \langle i,k'\rangle$ (in particular, $k\sim k'$);
        \item[(b)] $k'$ is minimal;
        \item[(c)] $\langle i,k'\rangle$ is quasi-malnormal.
\end{enumerate}
In particular, note that $C_{G_1}(ik)=C_{G_1}(ik')=\langle ik'\rangle$ and $N_{G_1}(\langle i, k\rangle)=N_{G_1}(\langle i, k'\rangle)=\langle i, k'\rangle$.
\end{lemme} 

\begin{proof}
Let $g=ik$. By assumption, this element is hyperbolic, and hence it fixes a pair of point $P=\lbrace p_{-}, p_{+}\rbrace$ in the boundary at infinity of the Bass-Serre tree (see Subsection \ref{trees}). Let us denote by $\mathrm{Stab}^{+}(A(g))$ the stabilizer of the axis $A(g)$ of $g$ that does not reverse the orientation, i.e.\ the set of elements of $G_1$ that fixes $p_{-}$ and $p_{+}$. This group $\mathrm{Stab}^{+}(A(g))$ decomposes as a semi-direct product $S\rtimes \langle t\rangle$ where $S$ denotes the pointwise stabilizer of the axis, and $t$ denotes an element acting on $A(g)$ as a translation, with minimal translation length (but beware that, a priori, $t$ does not belong to what we called the set of translations $T=I^{2}\setminus \langle 1\rangle$; in fact we shall prove that $t$ belongs to $T$). Denote by $\mathrm{Stab}(A(g))$ the stabilizer of $A(g)$. Observe that $\mathrm{Stab}^{+}(A(g))$ has index 2 in $\mathrm{Stab}(A(g))$. One deduces from the equality $igi=g^{-1}$ that $i$ swaps $p_{-}$ and $p_{+}$, i.e.\ that it reverses the orientation of $A(g)$, and thus one has $\mathrm{Stab}(A(g))=(S\rtimes \langle t\rangle)\rtimes \langle i\rangle$.

First, let us prove that $S$ is the trivial group. Note that $S$ is contained in a conjugate of $H$ or $K$ (because all edge groups in the Bass-Serre are conjugates of $H$ or $K$). Suppose towards a contradiction that $S$ is not trivial. We distinguish two cases.

\emph{First case.} Suppose that $S=\langle s\rangle$ for some involution $s\neq i$. In this case, we have $\mathrm{Stab}(A(g))=(\langle s\rangle \times \langle t\rangle)\rtimes \langle i\rangle$. Hence $i$ commutes with $s$, which contradicts the fact that any two distinct involutions of $G$ generate an infinite dihedral group.

\emph{Second case.} Suppose that $\vert S\vert > 2$. Recall that $S$ is contained in a conjugate of $H$ or $K$, and these groups are isomorphic to $\mathbb{Z}/n\mathbb{Z}$ for $n>2$, or to $\mathbb{Z}$, or to the infinite dihedral group $D_{\infty}$. Hence, $S$ contains an element $s$ of order at least 3. As $g$ normalizes $S$, it normalizes $\langle s\rangle$. But $g$ is hyperbolic, contradicting the fact that the normalizer of $\langle s\rangle$ in $G_1$ is elliptic in the Bass-Serre tree (by Lemma \ref{lemme}).

Hence, $S$ is trivial and $\mathrm{Stab}^{+}(A(g))$ is the infinite cyclic group $\langle t\rangle$. Let us prove that $t$ is a translation of the form $ik'$ for some involution $k'$ in $G_1$. There exists an integer $n\in\mathbb{Z}^{\ast}$ such that $g=ik=t^n$. Note that $it^n=k$ and thus $(it^n)^2=1$, i.e.\ $(iti)^n=t^{-n}$. It follows that $iti=t^{-1}$ since two elements of an infinite cyclic groups with the same $n$th power are equal. We see now that $(it)^2=1$. Since $it$ is non-trivial (as $i$ has order 2 and $t$ has infinite order), it has order 2. Let $it=k'$, so that $t=ik'$. 

We proved that $\mathrm{Stab}^{+}(A(g))=\langle ik'\rangle$ and $\mathrm{Stab}(A(g))=\langle ik'\rangle\rtimes \langle i\rangle=\langle i,k'\rangle$. Last, let us check that conditions \emph{(a)}, \emph{(b)} and \emph{(c)} hold.

\emph{Condition (a)}. Obviously, $g=ik$ belongs to $\mathrm{Stab}(A(g))=\langle i,k'\rangle$, and thus $\langle i,k\rangle$ is contained $\langle i,k'\rangle$. In particular, $ik$ is a power of $ik'$, which shows that $k$ and $k'$ are equivalent.

\emph{Condition (b)}. Let $\ell\in G_1$ be an involution such that $\langle i,k'\rangle\subset \langle i,\ell\rangle$ and prove that $\langle i,k'\rangle=\langle i,\ell\rangle$. Note that $i\ell$ acts hyperbolically on the Bass-Serre tree and that, necessarily, $i\ell$ and $ik'$ have the same axis. The inclusion $\langle i,k'\rangle\supset\langle i,\ell\rangle$ follows from the fact that $t=ik'$ has minimal translation length (by definition of $t$).

\emph{Condition (c)}. We shall prove that $\langle i,k'\rangle$ is quasi-malnormal in $G_1$. We proved above that $\langle i,k'\rangle$ coincides with the stabilizer of the axis of $A(ik)$, denoted by $\mathrm{Stab}(A(ik))$. This group is the stabilizer of a unique pair of points $P$ in the boundary at infinity of the Bass-Serre tree. If the intersection $g\mathrm{Stab}(A(ik))g^{-1}\cap \mathrm{Stab}(A(ik))$ is nontrivial for some $g\in G_1$, either it is infinite or it has order 2 (as $\mathrm{Stab}(A(ik))$ is dihedral infinite). If this intersection is infinite, the infinite cyclic group $g\langle ik'\rangle g^{-1}$ has a finite index subgroup that fixes the pair of points $P$. It follows that $g ik' g^{-1}$ fixes $P$ as well, i.e.\ that $ik'$ fixes $g(P)$. Since $P$ is unique, one has $g(P)=P$, and thus $g$ belongs to $\mathrm{Stab}(P)=\mathrm{Stab}(A(ik))=\langle i,k'\rangle$. Hence, $\langle i,k'\rangle$ is quasi-malnormal. Note however that this group is not necessarily malnormal \emph{a priori}: indeed, for every $g\in C_{G_1}(i)$, the involution $i$ is contained in the intersection of $\mathrm{Stab}(A(ik))$ and $g\mathrm{Stab}(A(ik))g^{-1}$, and these two sets are distinct in general (in other words, the axes of $ik$ and $gikg^{-1}$ intersect at a unique point which is fixed by $i$).\end{proof}


We now prove Proposition \ref{jesaispas0}.

\begin{prop}\label{jesaispas0b}
Let $G$ be a group in $\mathcal{C}$. Let $k\in G$ be an involution distinct from $i$. Suppose that $k$ does not belong to $\mathrm{Orb}_A(j)$ and that $k$ is minimal. Define \[G_1=\langle G,t \ \vert \ tit^{-1}=i, \ tjt^{-1}=k\rangle\] Then $G_1$ belongs to $\mathcal{C}$.
\end{prop}


\begin{proof}
Conditions \emph{(1)} and \emph{(2)} are satisfied by $G_1$, by Lemma \ref{pres_HNN}.

Before proving that condition \emph{(3)} holds in $G_1$, let us check that the assumptions of Lemma \ref{lemme} are satisfied. Define $H=\langle i,j\rangle$ and $K=\langle i,k\rangle$. By condition \emph{(5)} in $G$, since $k$ does not belong to $\mathrm{Orb}_A(j)$ and since $k$ is minimal, $K$ is quasi-malnormal in $G$. Moreover, we will prove that $\vert gHg^{-1}\cap K\vert \leq 2$ for every $g\in G$. First, observe that elements of $K$ have finitely many roots; more precisely, if $x\in K$ has infinite order, then there exists an integer $n_x$ such that $x$ has no $n$th root in $G$ for $n\geq n_x$ (indeed, if $x_n$ is an $n$th root of $x$, then $x$ belongs to $x_nKx_n^{-1}\cap H$ and thus $x_n$ belongs to $K$, as $K$ is quasi-malnormal in $G$). In contrast, note that each element of $gHg^{-1}$ of infinite order has an $n$th root for every $n>0$, by Lemma \ref{lemme0}. It follows that the intersection $gHg^{-1}\cap K$ cannot be infinite, and thus $\vert gHg^{-1}\cap K\vert \leq 2$. Hence Lemma \ref{lemme} applies and tells us that for every $x\in H$ of infinite order, the centralizer of $x$ in $G_1$ is contained in $G$.

Now, let us prove condition \emph{(3)} in $G_1$. For any involution $k\in I$, if $k\leq j$, then by definition there exist $m,n\in\mathbb{Z}\setminus\lbrace 0\rbrace$ such that $(ij)^m=(ik)^n$. Hence $ik$ centralizes $x=(ij)^m$, and as a consequence $ik$ belongs to $G$. Therefore $k$ belongs to $G$ and condition \emph{(3)}, which holds in $G$, ensures the existence of an element $a\in A\subset A_1$ such that $k=aja^{-1}$.


Let us prove condition \emph{(4)}, namely that $A_1=C_{G_1}(i)$ is malnormal in $G_1$. We will prove that for any involution $k\neq i$, if $k$ does not belong to $A_1$ then the intersection of $C_{G_1}(k)$ and $C_{G_1}(i)$ is trivial, or equivalently that the intersection of $C_{G_1}(ik)$ and $C_{G_1}(i)$ is trivial. So let $g$ be an element of $C_{G_1}(k)\cap C_{G_1}(i)$, and let us prove that $g=1$. We distinguish two cases: the case where $ik$ acts hyperbolically on the Bass-Serre tree of the splitting of $G_1$ as an HNN extension, and the case where $ik$ acts elliptically on the Bass-Serre tree (i.e.\ fixes a point in the tree).

\emph{First case.} If $ik$ is hyperbolic, then Lemma \ref{lemme2} (condition \emph{(c)}) shows that $C_{G_1}(ik)= \langle ik'\rangle$ for some minimal involution $k'$. Now, observe that any non trivial element of this group is of the form $i\ell$ for some involution $\ell\neq i$. Therefore, to prove that $C_{G_1}(ik)\cap A_1$ is trivial, we just have to prove that $i\ell$ does not commute with $i$. If they commuted, one would have $i(i\ell)=(i\ell)i$, so $(i\ell)^2=1$, which contradicts the fact that $i\ell$ has infinite order.

\emph{Second case.} Suppose now that $ik$ is elliptic. Since $g$ commutes both with $k$ and $i$, it belongs to $N_{G_1}(\langle i,k\rangle)$. By Lemma \ref{lemme}, $N_{G_1}(\langle i,k\rangle)$ is contained in a conjugate $xGx^{-1}$ of $G$ for some $x\in G_1$ (indeed, the normalizer of $\langle i,k\rangle$ normalizes the infinite cyclic subgroup $\langle ik\rangle$). Then, observe that $g$ belongs to $C_{G_1}(k)\cap C_{G_1}(i)$ if and only if $x^{-1}gx$ belongs to $C_{G_1}(x^{-1}kx) \cap C_{G_1}(x^{-1}ix)$. Recall that since condition \emph{(4)} holds in $G$, $A=C_{G}(i)$ is malnormal in $G$. But all involutions in $G$ are conjugate, and therefore the centralizer of any involution in $G$ is malnormal. It follows that $C_{G_1}(x^{-1}kx) \cap C_{G_1}(x^{-1}ix)$ is trivial (indeed, $x^{-1}kx$ does not belong to $C_G(x^{-1}ix)$, as $k$ does not belong to the centralizer of $i$). Hence $x^{-1}gx=1$, and thus $g=1$.

Then, prove condition \emph{(5)}. Let $k\in G_1$ be an involution that is not in the orbit of $j$ under the centralizer of $i$ in $G_1$. We shall prove that there is an involution $k'$ such that the following three conditions hold:
\begin{itemize}
        \item $\langle i,k\rangle\subset \langle i,k'\rangle$ (in particular, $k\sim k'$);
        \item $k'$ is minimal;
        \item $\langle i,k'\rangle$ is quasi-malnormal.
    \end{itemize}
Again, we distinguish two cases.

\emph{First case.} If $ik$ is hyperbolic, the result follows immediately from Lemma \ref{lemme2}.

\emph{Second case.} Suppose now that $ik$ is elliptic. First, let us state and prove a preliminary result.

\begin{claim}\label{preli_claim}If $ik$ is elliptic, then the following two assertions are equivalent:
\begin{enumerate}
\item $\langle i,k\rangle$ is contained in an edge group (i.e.\ $\langle i,k\rangle$ is contained in a conjugate of $\langle i,j\rangle$ in $G_1$);
\item $k$ belongs to the orbit of $j$ under the action of $A_1$.
\end{enumerate}
\end{claim}

\begin{proof}First, suppose that $k=aja^{-1}$ for some $a\in A_1$. Then $\langle i,k\rangle=a\langle i,j\rangle a^{-1}$ is an edge group. 

Conversely, suppose that $\langle i,k\rangle$ is contained in an edge group, i.e.\ that $\langle i,k\rangle\subset g\langle i,j\rangle g^{-1}$ for some $g\in G_1$, and let us prove that $k$ belongs to $\mathrm{Orb}_{A_1}(j)$. There is an involution $\ell\in \langle i,j\rangle$ such that $i=g\ell g^{-1}$. Observe that there is an element $h$ in the normalizer of $\langle i,j\rangle$ such that $\ell = hih^{-1}$; indeed, every involution of $\langle i,j\rangle$ is conjugate to $i$ or to $j$ by an element of $\langle i,j\rangle$, and by assumption there is an involution $s\in G$ such that $j=sis$, and so $s$ swaps $i$ and $j$ and normalizes $\langle i,j\rangle$. Hence, $i=g\ell g^{-1}=(gh)i(gh)^{-1}$, and it follows that $a_1=gh$ belongs to $A_1$. One has $\langle i,k\rangle\subset g\langle i,j\rangle g^{-1}=a_1\langle i,j\rangle a_1^{-1}$. Write $k=a_1k'a_1^{-1}$ for some involution $k'\in \langle i,j\rangle$. This element $k'$ is of the form $i(ij)^n$ for some $n\in\mathbb{Z}$, and thus $ik'=(ij)^n$. Note that $n$ is nonzero, as $k'\neq i$. In particular one has $k'\sim j$, and condition \emph{(3)} in $G$ implies that there is an element $a\in A$ such that $k'=aja^{-1}$. Therefore, $k=(a_1a)j(a_1a)^{-1}$ is in the orbit of $j$ under the action of $A_1$, which concludes the proof of the preliminary claim.
\end{proof}

Now, since we assume that $ik$ is elliptic and that $k$ does not belong to the orbit of $j$ under the action of $A_1$ by conjugation, Claim \ref{preli_claim} tells us that $\langle i,k\rangle$ is not contained in an edge group. First, suppose that $\langle i,k\rangle$ is contained in $G$. By condition \emph{(5)} in $G$, there is an involution $k'\in I$ such that the following hold:
    \begin{itemize}
        \item $\langle i,k\rangle\subset \langle i,k'\rangle$ (in particular, $k\sim k'$);
        \item $k'$ is minimal in $G$;
        \item $\langle i,k'\rangle$ is quasi-malnormal in $G$.
    \end{itemize}
One has to prove that these conditions are still true in $G_1$. The first point is obvious. For the second point, suppose that $\langle i,k'\rangle\subset \langle i,\ell\rangle$ for some involution $\ell\in G_1$. Then $\langle i,\ell\rangle$ is not contained in an edge group (otherwise $\langle i,k'\rangle$, and thus $\langle i,k\rangle$, would be contained in the same edge group!), so $\ell$ belongs to $G$ and thus $\langle i,k'\rangle = \langle i,\ell\rangle$ since we already know that $k'$ is minimal in $G$. Last, for the third point, note that the set of elements $g\in G_1$ such that $\vert g\langle i,k'\rangle g^{-1}\cap \langle i,k'\rangle\vert > 2$ is contained in $G$ by Lemma \ref{lemme}, because $i\ell$ is not contained in an edge group.

Now, suppose that $\langle i,k\rangle$ is contained in $gGg^{-1}$ for some $g\in G_1$. We will prove that $gGg^{-1}=aGa^{-1}$ for some $a\in A_1$, which means that one can simply conjugate by $a^{-1}$, and then everything boils down to the case described above, because $a^{-1}\langle i,k\rangle a = \langle i,a^{-1}ka\rangle\subset G$, and because $a^{-1}ka$ is not in the orbit of $j$ under $A_1$ since $k$ is not in the orbit of $j$ under $A_1$. 

So let us prove that $gGg^{-1}=aGa^{-1}$ for some $a\in A_1$. Since $\langle i,k\rangle$ is contained in $gGg^{-1}$, and since $i$ belongs to $G$, $i$ is contained in the pointwise stabilizer of the path from $v$ to $gv$, where $v$ denotes the vertex of the Bass-Serre tree fixed by $G$. This path contains an edge $f$ adjacent to $gv$. Let $e$ denote the edge $[v,t^{-1}v]$, whose edge group is $D_j=\langle i,j\rangle$, and observe that any edge adjacent to $v$ in the Bass-Serre tree is of the form $g'e$ or $g'te$ for some $g'\in G$. As a consequence, the edge $f$ is of the form $gg't^{\varepsilon}e$, with $\varepsilon\in\lbrace 0,1\rbrace$ and $g'\in G$. Let $z=gg't^{\varepsilon}$. Since $i$ fixes $ze$, it belongs to the group $zD_jz^{-1}$. The same argument as in the proof of Claim \ref{preli_claim} shows that $z=ah$ for some $a\in A_1$ and $h\in N_G(D_j)$ (in particular, $h$ belongs to $G$), i.e.\ $gg't^{\varepsilon}=ah$. We now distinguish two cases. 

If $\varepsilon=0$, then $gg'=ah$. As a consequence, $a^{-1}\langle i,k\rangle a$ is contained in $(a^{-1}g)G(a^{-1}g)^{-1}=(hg'^{-1})G(hg'^{-1})^{-1}=G$ since $hg'^{-1}\in G$. 
 
If $\varepsilon=1$, then $gg'=aht^{-1}$. As a consequence, $(th^{-1}a^{-1})\langle i,k\rangle (th^{-1}a^{-1})^{-1}$ is contained in $(th^{-1}a^{-1}g)G(th^{-1}a^{-1}g)^{-1}=(g'^{-1})G(g'^{-1})^{-1}=G$ since $g'\in G$. Let $k'=a^{-1}ka$, so that $(th^{-1}a^{-1})\langle i,k\rangle (th^{-1}a^{-1})^{-1}=th^{-1}\langle i,k'\rangle ht^{-1}$. Observe that since $th^{-1}\langle i,k'\rangle ht^{-1}$ is contained in $G$, necessarily $h^{-1}\langle i,k'\rangle h$ is contained in $D_j$. But $h$ normalizes $D_j$, and thus $\langle i,k'\rangle$ is contained in $D_j$. But $k'\notin \mathrm{Orb}_{A_1}(j)$, and every involution of $D_j\setminus \lbrace i\rbrace$ is in $\mathrm{Orb}_{A_1}(j)$ (indeed, every involution of $D_j\setminus \lbrace i\rbrace$ is of the form $\ell=i(ij)^n$ with $n\neq 0$, and thus $i\ell=(ij)^n$, i.e.\ $\ell\sim j$; therefore, by condition \emph{(3)} in the definition of the class $\mathcal{C}$, $\ell$ belongs to $\mathrm{Orb}_{A_1}(j)$). It follows that $k'=i$, contradicting the fact that $\langle i,k\rangle=a\langle i,k'\rangle a^{-1}$ is infinite.\end{proof}

Then, let us prove Proposition \ref{jesaispas}.

\begin{prop}\label{jesaispasb}
Let $G$ be a group in $\mathcal{C}$. Let $g,h\in G$ be two elements of order $n>2$. Suppose that the following conditions hold:
\begin{enumerate}
\item $\langle h\rangle$ contains no translation;
\item for every $y\in G$, $y\langle h\rangle y^{-1}\cap \langle g\rangle =\lbrace 1\rbrace$;
\item $\langle h\rangle$ is malnormal (i.e.\ $N_G(\langle h\rangle)=\langle h\rangle$).
\end{enumerate}
Define $G_1=\langle G,z \ \vert \ zgz^{-1}=h\rangle$. Then $G_1$ belongs to $\mathcal{C}$.
\end{prop}


\begin{proof}Conditions \emph{(1)} and \emph{(2)} are satisfied by $G_1$, by Lemma \ref{pres_HNN}. 

Before proving that condition \emph{(3)} holds in $G_1$, note that the assumptions of Lemma \ref{lemme} are satisfied (with $H=\langle g\rangle$ and $K=\langle h\rangle$). Hence Lemma \ref{lemme} applies and tells us that for every $x\in G$ of infinite order, if $x$ is not conjugate to a power of $h$, then the centralizer of $x$ in $G_1$ is contained in $G$.

Let us prove that condition \emph{(3)} holds in $G_1$. For any involution $k\in I$, if $k$ and $j$ are equivalent, then by definition there exist two integers $m,n\in\mathbb{Z}\setminus\lbrace 0\rbrace$ such that $(ij)^m=(ik)^n$. Hence $ik$ centralizes $x=(ij)^m$. Since $x$ belongs to $D_j$, the first assumption tells us that $x$ does not belong to a conjugate of $\langle h\rangle$ (indeed, every element of $D_j$ of infinite order is a translation, and in particular $x$ is a translation). As a consequence, the centralizer of $x$ in $G_1$ is contained in $G$. Hence $ik\in G$. Therefore $k$ belongs to $G$ and condition \emph{(3)}, which holds in $G$, ensures the existence of an element $a\in A\subset A_1$ such that $aka^{-1}=j$.

Then, the same argument as in the proof of the previous proposition shows that condition \emph{(4)} holds in $G_1$, namely that $A_1=C_{G_1}(i)$ is malnormal in $G_1$.

Last, prove condition \emph{(5)}. Let $k\in G_1$ be an involution that is not in the orbit of $j$ under the centralizer of $i$ in $G_1$. We shall prove prove that there is an involution $k'$ such that the following three conditions hold:
\begin{itemize}
        \item $\langle i,k\rangle\subset \langle i,k'\rangle$ (in particular, $k\sim k'$);
        \item $k'$ is minimal;
        \item $\langle i,k'\rangle$ is quasi-malnormal.
    \end{itemize}
Again, we distinguish two cases.

\emph{First case.} If $ik$ is hyperbolic, the result follows immediately from Lemma \ref{lemme2}.

\emph{Second case.} Suppose now that $ik$ is elliptic. We claim that $ik$ belongs to $G$. Indeed, note that $i$ is not contained in a conjugate of $G$ other than $G$, since edge groups in the Bass-Serre tree are torsion-free. As a consequence, if $k$ does not belong to $G$, then $ik$ is hyperbolic. Hence $k$ belongs to $G$ as well, and thus $ik$ is in $G$. By condition \emph{(5)} in $G$, there is an involution $k'\in G$ such that the following hold:
    \begin{itemize}
        \item $\langle i,k\rangle\subset \langle i,k'\rangle$ (in particular, $k\sim k'$);
        \item $k'$ is minimal in $G$;
        \item $\langle i,k'\rangle$ is quasi-malnormal in $G$.
    \end{itemize}
We have to prove that these conditions are still true in $G_1$. 1) Obvious. 2) Suppose that $\langle i,k'\rangle\subset \langle i,\ell\rangle$ for some involution $\ell\in G_1$. Suppose towards a contradiction that $\ell$ is not in $G$. Then $i\ell$ is hyperbolic, and it normalizes $\langle ik'\rangle$, which is not possible since the normalizer of $\langle ik'\rangle$ is elliptic in the Bass-Serre tree, by Lemma \ref{lemme}. Therefore $\ell$ lies in $G$, and the minimality of $k'$ in $G$ implies that $\langle i,k'\rangle = \langle i,\ell\rangle$. 3) Observe that $ik'$ does not belong to a conjugate of $\langle h\rangle$ (in $G$), as $\langle h\rangle$ contains no translation in $G$ by the first assumption of Proposition \ref{jesaispasb}. Then, by Lemma \ref{lemme}, the set of elements $g\in G_1$ such that $\vert g\langle i,k'\rangle g^{-1}\cap \langle i,k'\rangle\vert > 2$ is contained in $G$, and thus $\langle i,k'\rangle$ is quasi-malnormal in $G_1$, since it is quasi-malnormal in $G$.\end{proof}

\subsection{Stability under amalgamated free products}

The results in this subsection show that $\mathcal{C}$ is preserved under certain amalgamations. We start by proving a preliminary lemma which plays the same role as Lemma \ref{lemme2}.

\begin{lemme}\label{lemme_cas_amalgam}
Let $G$ be a group in the class $\mathcal{C}$ and let $E$ be a subgroup of $G$. Let $H$ be a group and let $\varphi : E \hookrightarrow H$ be an embedding, identifying $E$ with $\varphi(E)$. Define $G_1=G\ast_E H$. Suppose that one of the following assumptions hold:
\begin{itemize}
    \item either $E$ is trivial,
    \item or $E=\langle i\rangle$.
\end{itemize}
Let $k\in G_1$ be an involution. If $ik$ acts hyperbolically on the Bass-Serre tree of the splitting of $G_1$, then there exists an involution $k'\neq i$ in $G_1$ such that the following three conditions are satisfied (compare with condition \emph{(5)} in the definition of the class $\mathcal{C}$):
\begin{enumerate}
        \item[(a)] $\langle i,k\rangle\subset \langle i,k'\rangle$ (in particular, $k\sim k'$);
        \item[(b)] $k'$ is minimal;
        \item[(c)] $\langle i,k'\rangle$ is quasi-malnormal.
\end{enumerate}
In particular, note that $C_{G_1}(ik)=C_{G_1}(ik')=\langle ik'\rangle$ and $N_{G_1}(\langle i, k\rangle)=N_{G_1}(\langle i, k'\rangle)=\langle i, k'\rangle$.
\end{lemme}

\begin{proof}
As in the proof of Lemma \ref{lemme2}, write $\mathrm{Stab}(A(ik))=(S\rtimes \langle t\rangle)\rtimes \langle i\rangle$ for some element $t\in G_1$ with the same axis as $ik$ and with minimal translation length. The subgroup $S$ is contained in a conjugate of the edge group $E$. We shall prove that $S$ is the trivial group. If $E$ is trivial, there is nothing to do. Then, suppose that $E=\langle i\rangle$. If $S$ is not trivial, then $S=\langle s\rangle$ for some involution $s$. It follows that $\mathrm{Stab}(A(ik))=(\langle s\rangle \times \langle t\rangle)\rtimes \langle i\rangle$. In particular, $i$ commutes with $s$, which contradicts the fact that any two distinct involutions of $G$ generate an infinite dihedral group, and thus $S$ is trivial. Then, the rest of the proof is exactly the same as in the proof of Lemma \ref{lemme2}.
\end{proof}

\begin{prop}\label{prop5}
Let $G$ be a group in $\mathcal{C}$. Let $H$ be a group containing an involution $i'$. Suppose that $H$ satisfies the following conditions (that should be compared with the conditions in Definition \ref{class} with the same numbering):
\begin{enumerate}
    \item all involutions in $H$ are conjugate;
    \item any two distinct involutions of $H$ generate an infinite dihedral group;
    \item[(4)] $\langle i'\rangle$ is malnormal in $H$ (i.e.\ $N_H(\langle i'\rangle)=C_H(i')=\langle i'\rangle$);
    \item[(5)] for any involution $k\in H$, there is an involution $k'\in H$ such that the following hold:
    \begin{itemize}
        \item $\langle i',k\rangle\subset \langle i',k'\rangle$ (in particular, $k\sim k'$);
        \item $k'$ is minimal;
        \item $\langle i',k'\rangle$ is quasi-malnormal.
		\end{itemize}
\end{enumerate}
Then $G_1=G\ast_{i=i'}H$ belongs to $\mathcal{C}$.
\end{prop}


\begin{proof}
Conditions \emph{(1)} and \emph{(2)} are satisfied by $G_1$, by Lemma \ref{pres_amalgam}.

Condition \emph{(3)}: let $k$ be an involution in $G_1$ such that $k\sim j$, that is $(ik)^n=(ij)^m$ for some nonzero $n,m\in\mathbb{Z}$. Let $x=(ij)^m$ and observe that $ik$ centralizes $x$. But the centralizer of $x$ in $G_1$ is contained in $G$. Indeed, note that $x$ fixes a unique vertex in the Bass-Serre tree of the splitting of $G_1$, namely the unique vertex $v_G$ fixed by $G$, because $x$ cannot be contained in an edge group (as it has infinite order, whereas edge groups have order 2).

Condition \emph{(4)}: let $g\in G_1$ such that $gA_1g^{-1}\cap A_1\neq \lbrace 1\rbrace$. First, suppose that we have $\vert gA_1g^{-1}\cap A_1\vert > 2$. Then there is an element $a$ in $gA_1g^{-1}\cap A_1$ whose order is $>2$ (because two distinct involutions generate an infinite dihedral group). The same argument as above applies and shows that $v_G$ is the unique vertex fixed by $a$. Write $a=ga'g^{-1}$ for some $a'\in A_1$. Similarly, $v_G$ is the unique vertex fixed by $a'$. We have $ga'g^{-1}v_G=v_G$, so $a'g^{-1}v_G=g^{-1}v_G$ and hence by uniqueness $v_G=g^{-1}v_G$, which means that $g$ belongs to $G$. Then, if the intersection $gA_1g^{-1}\cap A_1$ has order 2, necessarily one has $gA_1g^{-1}\cap A_1=\lbrace i\rbrace$, and hence $gig^{-1}=i$, i.e.\ $g$ belongs to the centralizer of $i$, namely $A_1$. But we have $A_1=C_{G_1}(i)=C_G(i)\ast_{i=i'}C_H(i')=C_G(i)\ast_{i=i'}\langle i'\rangle=C_G(i)=A$. It follows that $g$ belongs to $G$. Conclusion: $A_1$ is malnormal in $G_1$.

Condition \emph{(5)}: let $k\in G_1$ be an involution that is not in the orbit of $j$ under $A_1$. If $ik$ is hyperbolic, Lemma \ref{lemme_cas_amalgam} shows that there exists a minimal involution $k'$ that satisfies the desired conditions. If $ik$ is elliptic, then $i$ and $k$ fix a common vertex $w$ (otherwise $ik$ would act hyperbolically on the Bass-Serre tree). Let us describe the fixed-point set of $i$. Let $v_G$ be the unique vertex fixed by $G$ and let $v_H$ be the unique vertex fixed by $H$. Observe that there is no vertex adjacent to $v_H$ oter than $v_G$ that is fixed by $i$: indeed, if $x$ is such a vertex, one has $x=hv_G$ for some $h\in H$. Since $i$ fixes $v_H$ and $x=hv_G$, it fixes the edge between these two vertices. But the stabilizer of this edge is $h\langle i\rangle h^{-1}$, and since $i$ is contained in the edge group one has $hih^{-1}=i$. Hence $h$ centralizes $i$, but the centralizer of $i$ in $H$ is simply $\langle i\rangle$ (by assumption), and thus $x=v_G$. It follows that the set of fixed points of $i$ is a star of radius $1$ centered at $v_G$. But $i$ belongs to $gv_H$ if and only if $g$ centralizes $i$ (same argument), and the vertices fixed by $i$ are precisely $v_G$ and the vertices of the form $av_H$ for some $a\in A$. As a consequence, $ik$ belongs to $G$ or to $aHa^{-1}$ for some $a\in A$. Up to replacing $k$ by $a^{-1}ka$ if necessary, one can assume without loss of generality that $ik$ belongs to $G$ or to $H$ (because $a^{-1}ka$ is not in the orbit of $j$ under $A_1$, as $k$ is not in the orbit of $j$ under $A_1$ by assumption). We distinguish the two cases.

\emph{First case.} Suppose that $ik$ is in $G$. Since condition \emph{(5)} holds in $G$ by assumption, we know that there is an involution $k'\in G$ such that the following hold:
    \begin{itemize}
        \item $\langle i,k\rangle\subset \langle i,k'\rangle$ (in particular, $k\sim k'$);
        \item $k'$ is minimal in $G$;
        \item $\langle i,k'\rangle$ is quasi-malnormal in $G$.
    \end{itemize}
We have to prove that these conditions are still true in $G_1$. 1) Obvious. 2) Suppose that $\langle i,k'\rangle\subset \langle i,\ell\rangle$ for some involution $\ell\in G_1$. Suppose towards a contradiction that $\ell$ is not in $G$. Then there are two options: either $\ell$ and thus $i\ell$ belong to $av_H$ for some $a\in A$, or $i\ell$ is hyperbolic. The first option is not possible since in that case one would have $ik=(i\ell)^p\in av_H$ for a nonzero integer $p$. The second option is not possible either, because $i\ell$ normalizes $\langle ik'\rangle$, and the normalizer of $\langle ik'\rangle$ is contained in $G$ (indeed, as the order of $ik'$ is $>2$, $v_G$ is the unique vertex fixed by $ik'$). Therefore $\ell$ lies in $G$, and the minimality of $k'$ in $G$ implies that $\langle i,k'\rangle = \langle i,\ell\rangle$. 3) The set of elements $g\in G_1$ such that $\vert g\langle i,k'\rangle g^{-1}\cap \langle i,k'\rangle\vert > 2$ is contained in $G$, and thus $\langle i,k'\rangle$ is quasi-malnormal in $G_1$ since it is quasi-malnormal in $G$.

\emph{Second case.} Suppose that $ik$ is in $H$. The same proof as above works, because we assume that $H$ satisfies condition \emph{(5)}.\end{proof}



We now deduce Proposition \ref{coro1} (see Corollary \ref{coro10} below) and Proposition \ref{coro2} (see Corollary \ref{coro20}) from the previous result.

\begin{co}\label{coro10}
Let $G$ be a group in $\mathcal{C}$. Consider the following group $H$: \[H=\Biggl\langle 
       \begin{array}{l|cl}
                        & abcde=1 \\
            a,b,c,d,e,u,v,w,x & a^2=b^2=c^2=d^2=e^2=1 \\
                        & ubu^{-1}=vcv^{-1}=wdw^{-1}=xex^{-1}=a  &                                             
        \end{array}
     \Biggr\rangle.\]
Then $G_1=G\ast_{i=a}H$ belongs to $\mathcal{C}$.
\end{co}

\begin{proof}
We just have to check that $H$ satisfies the assumptions of the previous proposition. First, consider the group given by the following finite presentation:\[O=\langle a,b,c,d,e \ \vert \ abcde=a^2=b^2=c^2=d^2=e^2=1 \rangle.\]This group is the orbifold fundamental group of a 2-dimensional orbifold whose underlying surface is a sphere, with five conical points of order 2. In particular, $O$ acts cocompactly and properly discontinuously by isometries on the hyperbolic plane $\mathbb{H}^2$. There are exactly five conjugacy classes of non-trivial elements (involutions) of $O$ fixing a point in $\mathbb{H}^2$, represented by $a,b,c,d,e$. Each involution has a unique fixed point. Two distinct involutions do not have the same fixed point, and their product has infinite order; in particular, they generate an infinite dihedral group. The group $H$ is then obtained from $O$ by adding four new letters conjugating $a$ and $b$, $a$ and $c$, and so on. Hence, all involutions are conjugate in $H$, which proves that condition \emph{(1)} holds in $H$.

Prove condition \emph{(2)}. Consider the splitting of $H$ whose underlying graph is a wedge of four circles, with $O$ as a vertex stabilizer and with $u,v,w,x$ as stable letters. Let $k,\ell$ be two involutions of $G_1$. Since $k$ and $\ell$ have finite order, they fix a vertex in the Bass-Serre tree of the previous splitting. We distinguish two cases. 

\emph{Case 1:} if they fix a common vertex, there exists $g\in G_1$ such that $k,\ell$ belong to $gOg^{-1}$. Hence, $g^{-1}kg$ and $g^{-1}\ell g$ generate an infinite dihedral group, and thus $k$ and $\ell$ generate an infinite dihedral group. 

\emph{Case 2:} if $k$ and $\ell$ do not fix a common vertex, then they generate an infinite dihedral group by Lemma \ref{dih}.

Then, prove condition \emph{(3)}. We want to prove that the subgroup generated by $i'=a$ is malnormal in $H$. First, let us prove that it is malnormal in $O$. Let $P$ be the unique point of $\mathbb{H}^2$ fixed by $a$, and let $g\in O$ be an element such that $\langle a\rangle\cap g\langle a\rangle g^{-1}$ is not trivial. We have $gag^{-1}=a$, and hence $g$ belongs to $\mathrm{Fix}(P)=\langle a\rangle$. It remains to prove that $\langle a\rangle$ is malnormal in $H$. First, consider the HNN extension $O'=\langle O,u \ \vert \ ubu^{-1}=a\rangle$. We just proved that $\langle a\rangle$ is malnormal in $O$; similarly, $\langle b\rangle$ is malnormal in $O$. Therefore, the fixed-point set of $a$ in the Bass-Serre tree of the splitting of $O'$ is reduced to the edge $e=[y,uy]$ where $y$ denotes the unique vertex fixed by $O$. The centralizer of $a$ in $O'$ fixes the barycenter of $\mathrm{Fix}(a)$, namely the middle of the edge $e$. Since $O'$ acts without inversion on the Bass-Serre tree, the centralizer of $a$ fixes the edge $e$ pointwise, and thus is contained in $O$. This argument shows that the centralizer of $a$ in $O'$ coincides with the centralizer of $a$ in $O$, that is $\langle a\rangle$. In other words, $\langle a\rangle$ is malnormal in $O'$. By repeating this argument three times, we see that $\langle a\rangle$ is malnormal in $H$.

Last, prove condition \emph{(5)}. The fact that condition \emph{(5)} holds in $O$ can be deduced easily from the action of $O$ on the hyperbolic plane. Then, let $k\in H$ be an involution. Consider the Bass-Serre tree of the splitting of $H$ defined above (see the proof of condition \emph{2} above). If $ik$ is hyperbolic, Lemma \ref{lemme_cas_amalgam} shows that there exists a minimal involution $k'$ that satisfies the desired conditions. If $ik$ is elliptic, then $i$ and $k$ fix a common vertex. One can therefore assume without loss of generality that $i$ and $k$ belong to $O$, which satisfies condition \emph{(5)}. Hence, there exists a minimal involution $k'\in O$ that satisfies the desired conditions in $O$. The fact that these conditions still hold in $H$ is straightforward (we refer the reader to the proof of the previous proposition for the details).\end{proof}

\begin{co}\label{coro20}
Let $G$ be a group in $\mathcal{C}$. Let us define the following groups, for $n\geq 2$:\[H_{\infty}=\Biggl\langle 
       \begin{array}{l|cl}
                        & abcdh=1 \\
            a,b,c,d,h,u,v,w & a^2=b^2=c^2=d^2=1 \\
                        & ubu^{-1}=vcv^{-1}=wdw^{-1}=a  &                                             
        \end{array}
     \Biggr\rangle,\]
     \[H_{2n-1}=\Biggl\langle 
       \begin{array}{l|cl}
                        & abcdh=1 \\
            a,b,c,d,h,u,v,w & a^2=b^2=c^2=d^2=h^{2n-1}=1 \\
                        & ubu^{-1}=vcv^{-1}=wdw^{-1}=a  &                                             
        \end{array}
     \Biggr\rangle,\]
     \[H_{2n}=\Biggl\langle 
       \begin{array}{l|cl}
                        & abcdh=1 \\
            a,b,c,d,h,u,v,w,x & a^2=b^2=c^2=d^2=h^{2n}=1 \\
                        & ubu^{-1}=vcv^{-1}=wdw^{-1}=xh^nx^{-1}=a  &                                             
        \end{array}
     \Biggr\rangle.\]
Let $H_m$ denote one of the previous groups (with $m\geq 3$, possibly infinite). Then $G_1=G\ast_{i=a}H_m$ belongs to $\mathcal{C}$. Moreover, the following conditions hold, for any element $g\in G$ of order $m$ (these conditions should be compared with the assumptions of Proposition \ref{jesaispas}):
\begin{enumerate}
    \item $\langle h\rangle$ contains no translation;
    \item for every $y\in G_1$, $y\langle g\rangle y^{-1}\cap \langle h\rangle=\lbrace 1\rbrace$;
    \item $\langle h\rangle$ is malnormal in $G_1$.
\end{enumerate}
\end{co}

\begin{rque}
The distinction between the presentations of $H_{2n-1}$ and $H_{2n}$ lies in the fact that, in the latter case, the cyclic group $\langle h\rangle$ contains an involution, which has to be conjugate to $a$.\end{rque}

\begin{proof}
In order to prove that $G_1$ is in $\mathcal{C}$, we just have to check that $H_m$ satisfies the assumptions of Proposition \ref{prop5}. The proof is similar to that of Corollary \ref{coro1}.

Then, let us explain why $\langle h\rangle$ satisfies the three conditions above.

\emph{Condition (1).} Clear: $h$ is a reduced product of four involutions by definition, and $h^k$ is a reduced product of $4k$ involutions for every integer $k$.

\emph{Condition (2).} Let $y\in G_1$. Denote by $v_G$ the unique vertex of the Bass-Serre tree of $G_1$ that is fixed by $G$, and denote by $v_H$ the unique vertex fixed by $H$. The only vertex fixed by $g$ is $v_G$ and the only vertex fixed by $h$ is $v_H$. Therefore, an element $x\in y\langle g\rangle y^{-1}\cap \langle h\rangle$ must fix the path joining $v_H$ to $yv_G$, and hence $x$ is trivial or a conjugate of $i$ (since the stabilizer of any edge in the Bass-Serre tree is conjugate to $\langle i\rangle$). But $\langle g\rangle$ and $\langle h\rangle$ are torsion-free, which implies that $x=1$.

\emph{Condition (3).} First, suppose that $m=\infty$. The group generated by $a,b,c,d,h$ can be viewed as the fundamental group of a 2-dimensional orbifold with one boundary component, with boundary group equal to $\langle h\rangle$, and four conical elements of order $2$, namely $a,b,c,d$. This description shows that $\langle h\rangle$ is malnormal because it is the stabilizer of a unique pair of points in the boundary at infinity of the orbifold group. It remains malnormal when one adds the new letters $u,v,w$ (the key point being that the order of $h$ is greater than the order of the edge groups in the Bass-Serre tree; indeed, edge groups have order 2 and $h$ has order at least 3 by assumption). Hence, $\langle h\rangle$ is malnormal in $H$. Let us prove that it is malnormal in $G_1$ as well. Let $g\in G_1$ be an element such that $g\langle h\rangle g^{-1}\cap \langle h\rangle$ is nontrivial. Since $v_H$ is the only vertex fixed by $\langle h\rangle$ in the Bass-Serre tree, $g$ fixes $v_H$ as well, and hence $g$ belongs to $H$. Since we proved that $\langle h\rangle$ is malnormal in $H$, it is also malnormal in $G_1$.

Then, suppose that $2 < m<\infty$. The argument showing that $\langle h\rangle$ is malnormal in $H$ is very similar to the argument above: now $\langle h\rangle$ does not fix a unique pair of points in the boundary at infinity of the orbifold group, but it fixes a unique point in the hyperbolic plane on which the orbifold group acts by isometry, and this implies that $\langle h\rangle$ is malnormal in $H$. Then, the argument showing that $\langle h\rangle$ is malnormal in $G_1$ is unchanged.\end{proof}

\renewcommand{\refname}{Bibliography}
\bibliographystyle{alpha}
\def\cprime{$'$} \def\cprime{$'$}

\vspace{8mm}

\textbf{Simon André}

Institut für Mathematische Logik und Grundlagenforschung

Westfalische Wilhelms-Universität Münster

Einsteinstraße 62

48149 Münster, Germany.

E-mail address: \href{mailto:sandre@wwu.de}{sandre@wwu.de}

\vspace{5mm}

\textbf{Katrin Tent}

Institut für Mathematische Logik und Grundlagenforschung

Westfalische Wilhelms-Universität Münster

Einsteinstraße 62

48149 Münster, Germany.

E-mail address: \href{mailto:tent@wwu.de}{tent@wwu.de}

\end{document}